\documentclass[11pt]{article}
\usepackage{graphicx}
\usepackage{subfigure}

\usepackage{xcolor}
\usepackage{amsthm,bm}
\usepackage{soul}
\usepackage{comment}
\usepackage{amsmath,amssymb}
\def\M{{\mathcal M}}

\def\ds{\displaystyle}

\def\D{{\cal D}}
\def\N{{\cal M}}
\def\V{{\mathcal V}}
\def\W{{\mathcal W}}
\def\w{{\mathfrak{w}}}
\def\FW{{\mathfrak{W}}}

\def\M{{\mathcal M}}

\def\R{{\mathbb R}}

\def\implies{\Longrightarrow}

\def\T{{\cruk}}

\def\rr{\mathbb{R}}

\def\bel{\begin{equation}\label}
\def\eeq{\end{equation}}
\def\sqr#1#2{\vbox{\hrule height .#2pt
\hbox{\vrule width .#2pt height #1pt \kern #1pt
\vrule width .#2pt}\hrule height .#2pt }}
\def\square{\sqr74}

\def\bega{\begin{array}}
\def\enda{\end{array}}
\def\begi{\begin{itemize}}
\def\endi{\end{itemize}}

\def\cruk{{\mathfrak T}}

\newtheorem{thm}{Theorem}[section]
\usepackage{yfonts}
\newtheorem{cor}{Corollary}[section]
\newtheorem{lma}{Lemma}[section]
\newtheorem{prop}{Proposition}[section]
\newtheorem{remark}{Remark}[section]
\newtheorem{definition}{Definition}[section]

\title{
A geometrically based criterion\\ to avoid infimum-gaps in Optimal Control
 }
\author{Michele Palladino and Franco Rampazzo}

\begin{document}

\maketitle


\abstract{
 In optimal control theory the expression {\it infimum gap} means  a stricly negative  difference between  the infimum value of a given minimum problem  and  the infimum value of a new problem obtained by the former by {\it extending} the original family $\V$ of controls to a larger family  $\W$.   Now, for some classes  of domain-extensions 
 --like  convex relaxation or impulsive embedding of unbounded control problems-- the  {\it normality} of an extended minimizer  has been shown to be  sufficient for the avoidance of an {infimum gaps}. A natural issue is then the search of a general hypothesis under which the  criterium {\it `normality implies no gap'} holds true.  We prove   that, far from being a peculiarity of  those specific   extensions and from requiring the convexity of the extended dynamics,   this criterium  is valid provided the original family $\V$ of controls  is {\it abundant} in the extended family $\W$. %
 {\it Abundance}, which is stronger than  the mere $C^0$-density of the original trajectories  in the set of extended trajectories, is  a dynamical-topological  notion introduced  by J. Warga, and is  here utilized in a `non-convex'  version which, moreover, is adapted to differential manifolds. 
 To get  the main result, which is based on set separation arguments, we prove  an  open mapping result valid for Quasi-Differential-Quotient  (QDQ) approximating cones, a notion of `tangent cone' resulted as a peculiar specification of H. Sussmann's Approximate-Generalized-Differential-Quotients (AGDQ) approximating cone.

\section{Introduction}
 One of the main reason for enlarging the domain of a minimum problem relies on the  
 aim of establishing the existence of at least one solution.
Actually, domain extension is a quite common and variously motivated  practice,  in particular in the Calculus of Variations and in  Optimal Control. Of course, a  crucial requisite of such a domain enlargement consists in the {\it density } of the original problem in the new one: the extended minimum  should  be  approximable by processes of the original problem.
However, because of the presence of a final target,  even  a  dense extension of the  domain may result in  the occurrence of an  {\it infimum gap}: namely, it can happen that  the infimum value of the original problem is  strictly greater than the infimum value of the extended problem. This might be undesirable in many respects, for instance in the convergence  of numerical schemes  as well as in the identification of the value function via Hamilton-Jacobi equations.
This raises a natural question: how can one avoid this gap phenomenon?
 A sufficient condition for gap avoidance seems  to emerge from   investigations by J. Warga \cite{warga2, warga, warga1, warga3} and from some other more recent papers  \cite{MRV, PRCDC, PV1, PV2, PV3}, dealing with some particular cases: this criterion is the so-called  {\it normality } of minimizers. Therefore, the mentioned question can be turned into the following one:
 
 {\bf Q.}{\it  Under which hypotheses on a general optimal control problem  normality
is sufficient for gap-avoidance?}

In order be more precise,   let us briefly sketch the abstract setting  of  our optimal control problem. The {\it state} variable $y$ will  range  on a Riemannian manifold ${\mathcal M}$, while
 the control maps $v(\cdot)$ will   belong to  an {\it original} family $\V\subset\W:= L^1([0,S],\FW)$ ---where   $\FW$ is a subset of a metric space-- or to a   larger set $\W$, which  will be called the {\it extended} family of controls. 
Given an initial state $\bar y\in {\mathcal M}$ and a time interval $[0,S]$, we will consider the control system
\[
(E)\,\qquad\qquad
\left\{\begin{array}{l}\displaystyle \ds\frac{dy}{ds}(s)={f}(s, y(s),w(s))
\\y(0)=\bar y,
\end{array}\right. 
\]
and, for every $w\in\W$, we will use  $y[w]:[0,S]\to {\mathcal M}$ the corresponding (supposedly unique)   solution.
 The {\it original optimal control problem}  is defined as 
 \[
(P)_{\mathcal{V}}\,\qquad
\hbox{\it Minimize}\Big\{ \, h\big(y[v](S)\big) \quad |\quad
v\in\V,\quad y[v](S)\in \T \Big\},\quad 
\]
where  the  {\it cost function} $h:{\mathcal M}\to\rr$ is continuous, and
 $\T\subset{\mathcal M}$ is a closed set called  {\it target.} 

Replacing the family of controls $\V$  by the larger set  $\W$, one  obtains the  {\it  extended   optimal control problem:} 
  \[
(P)_{\mathcal{W}}\,\qquad
\hbox{\it Minimize}\Big\{ \, h\big(y[w](S)\big) \quad |\quad
w\in\W,\quad y[w](S)\in \T \Big\}.\quad 
\]

We will assume the existence of a local minimum   for the extended problem, namely  a control   $\hat w\in\W$ such that, for some $C^0$ neighbourhood $\mathcal{O}$ of $y[\hat w]$   , 
 $
 h(y[\hat w](S))\leq h(y[ w](S))
 $
 for all  $w\in\W$ such that  $y[w](S)\in\T$ and $y[w]\in \mathcal{O}$. The non-occurrence of infimum gaps means that  the original infimum value is unaffected by the introduction of the extended controls, namely 
 $$
h(y[\hat w](S))= \inf\Big\{  h(y[w](S)) \quad |\quad v\in \V,\quad y[v](S)\in\T,\ y[v]\in \mathcal{O} \Big\}
  $$
    for all sufficiently small neighbourhoods  $\mathcal{O}$ of $y[\hat w]$
  
 If, on the contrary,  there exists a  neighbourhood $\mathcal{O}$ such that  
$$
h(y[\hat w](S)) < \inf \Big\{  h(y[w](S)) \quad |\quad v\in \V,\quad y[v](S)\in\T,\  y[v]\in  \mathcal{O}\Big\},
$$ one says that  the optimal control  problem {\it  has an infimum gap at $y[\hat w]$.} Obviously, via the usual reductions, one can formulate a  notion of infimum gap for a general  Bolza problem --where an integral cost is involved as well-- .

For problems defined on Euclidean spaces and such that  the extended dynamics is convex, an insightfull investigation of the gap question and its relation with normality  was carried out by J.Warga (see e.g. \cite{warga}). More recently  two  specific classes of domain extentions  --still assuming the convexity of the extended dynamics-- have been studied in  \cite{MRV, PV1, PV2, PV3}. As mentioned above, these investigations share the fact that a certain condition turns out to be  necessary for  the gap occurrence:
 \vskip0.3truecm
 {\bf(A)} {\it There is an infimum-gap only  if the minimum of the extended problem is an abnormal extremal}.  \ \ \ \footnote{Equivalently: if the minimum is normal (=not abnormal) there is no gap.}
   \vskip0.3truecm  
   
Since any version of the PMP states that `an optimal process $(\hat y, \hat w):=(y[\hat w],\hat w)$ for the extended problem  $(P)_\W$ is an  extremal', in order for {\bf(A)} to have a precise meaning one has to specify  which version   of the   Pontryagin Maximum Principle (PMP) one  is considering. In turn, this is equivalent to specify which kind of {\it approximating cones} we are going to utilize for both the{ \it reachable set} and
the target $\cruk$. Actually, for this goal we shall introduce a generalized differential called {\it Quasi Differential Quotient} (QDQ)   (Def. \ref{qdq})  \footnote{A QDQ is a special case of Sussmann's Approximate Generalized Differential Quotient \cite{sus}.} and the associated notion of  QDQ {\it  approximating cone} (Def. \ref{ApprCone}). Let us say immediately that this choice is perhaps the most important step for the validity of the main result. And while it is impossible at this stage to give an exhaustive  description of what   QDQ  approximating cones are, let us 
point out that, on the one hand, they are {\it sufficiently small} for a   fixed point theorem to hold true and, on the other hand, they are {\it enough large} to allow the utilization of the notion of {\it abundance}, which, as we shall see, proves crucial for normality to imply no gap.

 This said, let us give the notion of {\it extremal}. For simplicity, we consider here only the case when the state ranges on a Euclidean space. Moreover, if $C\subset\R^n$ is a cone, we use $C^\perp$ to denote the {\it polar cone of $C$}, namely the set of linear  forms
 $\lambda\in({\R^n})
 ^*$ such that $\lambda\cdot c\leq 0$ for all $c\in C$.

  {\begin{definition}[Extremal]\label{extremalityn} Consider a control $\hat w\in \W$  and the corresponding trajectory $\hat y:=y[\hat w]$. Assume  that  $\hat y(S)\in\cruk$,
   and  let   ${C}$ be  a QDQ approximating cone    of the target $\cruk$ at $\hat{y}(S)$.
We say that the process $(\hat{y},\hat{w})$ is  an {\rm  extremal} (with respect to $h$ and  $C$)   if  there exist an absolutely continuous (adjoint) path
${\lambda}\in  W^{1,1}([0,S];\,(\R^n)^*)$
  and  a {\rm cost multiplier}  $\lambda_c\in\{0,1\}$ such that $ (\lambda
  ,\lambda_c)\neq 0$ and 
\begin{itemize}
\item[$(i)\;$]$\ds \frac{d{\lambda}}{ds} = -\lambda\cdot \frac{\partial f}{\partial x}(s,{\lambda}(s),\hat w(s))$
 \item[$(ii)\;$]  
 $\ds \max_{\w\in\FW}  {\lambda}(s)\cdot{f}(s,\hat{y}(s),\w)={\lambda}(s)\cdot{f}(s,\hat{y}(s),\hat{w}(s))$ \  \ 
 a.e. $s\in [0,S]$;
\item[$(iii)\;$]  $ {{\lambda}}(S)\in -\lambda_c \nabla h(\hat{y}(S))- {C}^{\perp}$.
 \end{itemize}
Furthermore, we say that  an extremal $(\hat y, \hat w)$ is {\rm normal} if for every  choice of the pair $({\lambda},\lambda_c)$ one has  $\lambda_c=1$.  We  say that  an extremal $(\hat y, \hat w)$ is {\rm abnormal} if it is not normal, namely, if exists  a choice of $(\lambda,\lambda_c)$ with $\lambda_c=0$.
\end{definition} 
   
 As mentioned above, in \cite{PV1, PV2, PV3}, where  the original set of controls  $\V$ was embedded in the   set $\W$ of {\it relaxed} controls, it has been shown the validity of criterion {\bf (A)}, that is: {\it if the optimal process $(\hat y, \hat w)$ is a {\it normal extremal}, then an infimum gap cannot occur at $\hat y$.} An akin  result has been achieved in \cite{MRV}, where the system is control-affine and the  original set $\V$ comprises {\it unbounded }controls  ranging in a convex cone.  In that case  a space-time, impulsive, extension is considered, namely the larger set of trajectories corresponding to $\W$ comprises  space-time paths  which are allowed to evolve along fixed time  directions.\footnote{It is well-known that under  commutativity hypotheses  these paths could be regarded as measure, while the measure-theoretical approach is unfit for  non-commutative  problems,
 see e.g. \cite{BR},\cite{MiRu}.}
 
It is worth noticing that   in both the investigated  cases the original set of trajectories  is {\it dense} in the set of  extended trajectories, when the latter is endowed with  $C^0$ topology. So, one might conjecture that criterion {\bf (A)} is still valid as soon as the trajectories corresponding to  $\V$ are  dense in the set of trajectories corresponding to $\W$. 
 In fact, {\it this is not the case} , as shown by the simple example  in  Appendix \ref{ex_suss}.

Hence, a condition stronger than density is needed. For this goal we introduce Kaskoz' version of J. Warga's notion of $\V$ being  {\it abundant}  in $\W$ (Def. \ref{abundant}).  This condition strengthens  density  by requiring  that the trajectories of the  extended system's convexification  are uniformly approached by trajectories of the original system.

We will further  extend  the notion of abundant subfamily $\V\subset\W$ to systems defined on manifolds and to  fairly general classes of controls (which are merely required to belong to a metric space).  Then, aiming to 
express normality of extended trajectories in geometric terms,  
 we invoke   local set separation of the target from the original reachable set. 
 
 A crucial result for the achievement of the main theorem consists in showing that,  with this notion of abundance, {\it  every needle-variational cone  $ \bf C$ at  $\hat y$ corresponding to the enlarged domain $\W$ is also a QDQ approximating cone to the original reachable set} (Theorem  \ref{apprpicc}). 
 
The next step consists in showing that the  local set separation of the target from the original reachable set implies the linear separability between a  QDQ approximating cone to the target and the above mentioned needle-variational cone $\bf C$ (Theorem \ref{GEOM}).  This is exactly the point where the choice of QDQ approximating  cones --rather then other more classical cones, e.g. Boltiansky cones-- plays essential. By expressing  this linear separation in terms of adjoint paths,  one finally gets the  main result of the paper (Corollary \ref{MP_INF_GAP}), where, under the abundance hypothesis, statement {\bf(A)} is turned into an actual theorem.
 
 In Section \ref{Ex_Imp} we apply the main theorem to nonlinear  systems  whose dynamics are neither bounded nor convex.
  Finally, since normality cannot be verified {\it a priori},  it is important to find sufficient conditions on the data guaranteeing that all minimizers are normal. This is what is provided by Theorem \ref{normsuff},  where a  directly  verifiable criterion for normality is proved to hold true in the general setting. }}

\subsection{Basic notions and notation}\label{basicss}

\subsubsection{Linear spaces, manifolds}

Let $E$ be a real linear space, and let us use $E^*$ to denote the algebraic dual of $E$.  If $\langle\cdot,\cdot\rangle$ is a given scalar product on $E$,\footnote{By {\it scalar product} we mean a  positive definite, symmetric,  bilinear form.} we will use $|\cdot|$
 to denote the norm associated with $\langle\cdot,\cdot\rangle$, namely, for every $e\in E$ we set   $
 |e|=\sqrt{{\langle e,e \rangle}}. 
 $
For every $e\in E $ and every real number $r\geq 0$ let us use $e+B_r$ to denote the closed ball of center $e$ and radius $r$, namely $e+B_r=\{e+f\ | \ |f|\leq r\}$.  When $e=0$ we will  write $B_r$ instead of $0+B_r$

  If $E_1,E_2$ are real  linear  spaces  and $L\in Lin(E_1,E_2)$,  we shall use 
$L\cdot e $ to denote the element of $E_2$ coinciding with the image of $e\in E_1$. 
We will use the symbol $\cdot$ also to mean  duality. Furthermore, if $\lambda\in E_2^*$ and and $L\in Lin(E_1,E_2)$, sometimes we will use the notation $\lambda \cdot L$
to mean {``\it the element of $E_1^*$ coinciding the image of $\lambda$ through the dual map of $L$.''}
While it doesn't generate any confusion, this promiscuous use of the  notation ``$\cdot$" makes the writing 
$\lambda\cdot L \cdot e$ unambiguous, for one has
$ (\lambda\cdot L) \cdot e = \lambda\cdot (L \cdot e)$ for all 
$(e,\lambda)\in E_1\times E_2^*$.


By saying that  $\Big({\mathcal M},{\langle\cdot,\cdot\rangle}\Big)$   is a  Riemannian   differentiable manifold we will mean that 
${\mathcal M}$ is a differential manifold and  $  \langle\cdot,\cdot\rangle$ is a Riemannan metric.  For every $x\in\N$ and $e,f\in T_x\N$, 
$ \langle e, f \rangle_x  $ will denote the corresponding  scalar product of $e,f$, and $| e|_x:=\sqrt{{\langle e,e \rangle}_x}$  will be called the  {\it norm} of $e$.
  We will often omit the subscript and we will  write $\langle e, f\rangle$ and $|e|$ instead of   $\langle e, f\rangle_x$ and $|e|_x$.
 
We will use $d$ to denote the distance induced on $\mathcal{M}$ by  $\langle\cdot,\cdot\rangle$. We recall that, if $x_1,x_2\in\mathcal{M}$,  the {\it distance}  $d(x_1,x_2)$ is  defined as the minimum among the $\langle\cdot,\cdot\rangle$-lenghts of the absolutely continuous  curves having $x_1,x_2$ as end-points.
For any $x\in\N$ and any $r\geq 0$, we will use $\mathcal{B}[x,r]$ to denote the closed  ball of radius $r$ and center $x$, i.e. $\mathcal{B}[x,r]:= \left\{ y\in\N \ |  \ d(x,y)\leq r\right\}$.


 \

\subsubsection{Cones}
Let $E$   be  a real  linear space. A subset $K\subset E$   is a {\it cone} if $\alpha k\in K$  for all $(\alpha,k)\in [0,+\infty[\times K$. If $A\subset E$ is any subset, we use $span^+A$ to denote the smallest convex cone containing $A$.  Let us introduce a notion of  {\it transversality} for cones.

The idea of a non-trivial intersection between cones, which plays essential in  set-separation results like Theorem \ref{teoteo} below,  is made  formal is made formal by  the following  notion of {\it tranversality}:

\begin{definition}\label{trancon}
Let $E$ be a linear space and let   $ K_1, K_2\subseteq E $ be convex cones. We say that\begin{enumerate}
\item $ K_1$ and $ K_2 $ are {\em transverse}, if $
K_1-K_2:=\big\{k_1-k_2 ,\ (k_1,k_2)\in  K_1\times K_2\big\} = E$ ;\\
\item  $ K_1$ and $ K_2$  are {\em strongly transverse},
if they are transverse  and $ K_{1}\cap K_{2} \supsetneq\{0\}$.
\end{enumerate}
\end{definition}

Transversality differs from strong transversality only when $ K_1$ and $K_2 $ are  complementary subspaces:
\begin{prop}\label{teo2} Let $E$ be a linear space, and let $ K_1 , K_2\subseteq E $ be convex cones. Then  $ K_1 , K_2 $ are transverse if and only if 
either  $ K_1 , K_2 $ are strongly transverse  or $ K_1 , K_2 $
are linear subspaces  such that $ K_1\oplus K_2=E$.
\end{prop}

\begin{definition}\label{polar} Let $E$ be a finite-dimensional   linear space, and let $E^*$ be  its dual space.  For any subset  $A\subset E$, the (convex) cone 
$$
A^{\bot} \doteq \{ p\in E^* :\quad  p\cdot w \leq 0 \quad \forall \;
w\in A \}\subseteq E^*
$$
(where the symbol $\cdot$ denotes duality) will be called  the {\rm polar cone} of $A$. 
\end{definition}

The transversality of two cones  is  equivalent to their linearly separability. More precisely:
\begin{prop}\label{teo3}Two  convex cones $K_1$ and $K_2$ are not transverse  if and only if \linebreak $K_1^\bot\cap K_2^\bot \backslash \{0\} \neq \emptyset$ , namely there exists a  linear form $\lambda\neq 0$
 such that
  $$ \lambda \cdot k_1\geq 0 \  \forall k_1\in K_1
\quad \hbox{ and}\quad  \lambda \cdot k_2\leq 0 \  \forall k_2\in K_2 .$$
In this case one also says that $K_1$ and $K_2$ are {\rm linearly separable}.
\end{prop}

\subsubsection{Scorza-Dragoni points}
\begin{definition}[Scorza-Dragoni point] Given a compact  set $X\subset \mathcal{M}$ and an interval $[a,b]\subseteq \rr$,  $a<b$, let us consider  a function $\varphi : [a,b]\times X\rightarrow \rr^{n}$ verifying 
\begin{itemize}
\item[i)] $[a,b]\ni s\mapsto \varphi(s,y)\in \rr^n$ is measurable for each $y\in X$;
\item[ii)] $ X \ni y\mapsto \varphi(s,y)$ is continuous for each $s\in [a,b]$,
\end{itemize}
 We say that $\bar{s}\in [a,b]$ is a Scorza-Dragoni point for $\varphi$ if, for all $y\in X$,
 \begin{equation}
 \lim_{r\rightarrow 0}\, \lim_{\delta\rightarrow 0}\frac{1}{\delta}\int_{\bar{s}}^{\bar{s}+\delta} \Lambda_{r}(s,y)\,ds =0
 \end{equation}
 where
 \begin{equation}
 \Lambda_{r}(s,y):=\sup_{x\in X, \; |x-y|\leq r}d\left( \varphi(s,x),\varphi(\bar{s},y) \right)
 \end{equation}

We shall use $\mathrm{SD}\left\{\varphi \right\}$ to denote the {\rm set of all the Scorza-Dragoni points for the function $\varphi$}. 
\end{definition}
Notice in particular  that, if  $s\in \mathrm{SD}\left\{\varphi \right\}$,  one has \begin{equation}\label{SD_4}
\lim_{x\rightarrow y,\, \delta\searrow 0} \varphi(s+\delta,x)=\varphi(s, y),
\end{equation}
for any  $y\in X$.
The importance of Scorza-Dragoni points relies on the fact they they form a full measure set \cite{S}:
 
 \begin{thm}[Scorza-Dragoni] The set of all the Scorza-Dragoni points { of a Caratheodory  function  $\varphi : [a,b]\times X\rightarrow \rr^{n}$} has  measure equal to  $b-a$.\end{thm}

 \section{Set separation and open mappings}
\subsection{Quasi Differential Quotients}

 In order to state the  set-separation theorem (Th. \ref{teoteo}) we need the  notion of {\it Quasi Differential Quotients approximating cone} to a set $\mathcal{E}$ at a point of its boundary. For this purpose let us introduce the notion of Quasi Differential Quotient, which in turn is a particular case of Sussmann's  Approximate Generalized Differential Quotient \cite{sus}. The corresponding set-separation theorem is based on an Open Mapping  we prove below.

Let us  recall the notion of  Cellina continuously approximable set-valued function, which  is the  building block in the definition of Approximate Generalized Differential Quotient.
\begin{definition}[CCA]\label{cca}{Let  $F : \R^N \rightsquigarrow \R^n$   be a 
set-valued map. We say that $F$ is a { \rm  Cellina continuously approximable (CCA) set-valued map} if,  for any compact set $K\subset \R^N$,
\begin{itemize}
\item the restriction of $F$ on $K$ has compact graph, that is, the set $\mathrm{Gr}(F_{|_K}):=\{ (x,y)\in K\times \R^n :\; y\in F(x)\}$ is compact, and 
\item there exists a  sequence of single-valued, continuous maps $f_{k}:K\rightarrow \R^n$, $k\in\mathbb{N}$, such that the following condition holds:  for every open set $\Omega$ satisfying $\mathrm{Gr}(F_{|_K})\subset \Omega$, there exists $k_{\Omega}$ such that $\mathrm{Gr}\, (f_{k}):=\{ (x,y)\in K\times \R^n :\; y\in f_k(x)\} \subset \Omega$ for every $k\geq k_{\Omega}$.
\end{itemize}
}
\end{definition}

We will say that a function  $\rho : [0;+\infty[\to[0;+\infty] $  is a {\it a pseudo-modulus}  if it  is monotonically nondecreasing and
$\lim_{s\to 0^+}\rho(s) =\rho(0)= 0$. 
 We call {\it modulus} a  pseudo-modulus taking values in $[0,+\infty[$.

\begin{definition}[AGDQ]\label{agdq}
Assume that $F : \R^N \rightsquigarrow \R^n$       is a 
set-valued map, $(\bar \gamma,\bar y) \in \R^N\times\R^n  $,   $\Lambda\subset Lin\{ \R^N, \R^n\} $   is a compact set,  and $\Gamma\subset\R^N$  is any  subset.
We
say that $\Lambda$ is an {\em Approximate Generalized Differential Quotient  (AGDQ) of F at  $(\bar \gamma,\bar y)$ } in the direction of $\Gamma$  if there
exists a pseudo-modulus  $\rho$
having the property that
\begin{itemize}
\item[(*)]
 for every $\delta>0$ such that $\rho(\delta)<+\infty$, there exists  CCA  set-valued map \linebreak $A^\delta:{\left(\bar{\gamma}+B_{\delta}\cap\Gamma\right)}\to Lin\{ \R^N, \R^n\} \times \R^n$ 
  such that
 $$
 \inf_{L'\in\Lambda}|L - L'|\leq \rho(\delta), 
 \quad |h(\gamma)|\leq \delta \rho(\delta), \quad\hbox{and}\
  \bar y +  L\cdot(\gamma-\bar \gamma)  + h\in F(\gamma) \ \ \footnotemark
 $$\footnotetext{  Here $|\cdot|$ denotes the operator norm, namely $\displaystyle |M| = \sup_{|v|=1} |M\cdot v|$, for every linear operator $M\in Lin(\rr^N,\rr^n)$. }
whenever $\gamma
\in \bar\gamma+B(\delta)\cap\Gamma$ and $(L,h)\in A^\delta(\gamma)$.
 \end{itemize}
\end{definition}



{

We will use a subclass of AGDQs, which we call the  {\it Quasi Differential Quotients}. Their main property consists in the validity  of  an actual, {\it not punctured},  open mapping theorem (see Theorem \ref{om}below).

\begin{definition}[QDQ]\label{qdq}
Assume that $F : \R^N \rightsquigarrow \R^n$       is a 
set-valued map, $(\bar \gamma,\bar y) \in \R^N\times\R^n  $,   $\Lambda\subset Lin\{ \R^N, \R^n\} $   is a compact set,  and $\Gamma\subset\R^N$  is any  subset.
We
say that $\Lambda$ is a {\em Quasi Differential Quotient  (QDQ) of F at  $(\bar \gamma,\bar y)$ } in the direction of $\Gamma$  if there
exists modulus $\rho:[0,+\infty[\to [0,+\infty[$
having the property that
\begin{itemize}
\item[(*)]
 for every $\delta>0$  there exists a continuous map  
$(L_\delta,h_\delta):{\left(\bar{\gamma}+B_{\delta}\cap\Gamma\right)}\to Lin\{ \R^N, \R^n\} \times \R^n$ 
  such that
$$\min_{L'\in\Lambda}|L_\delta(\gamma) - L'|\leq \rho(\delta), 
 \quad |h_\delta(\gamma)|\leq \delta \rho(\delta), \quad\hbox{and}\
  \bar y +  L_\delta\cdot(\gamma-\bar \gamma)  + h_\delta(\gamma)\in F(\gamma),$$
whenever $\gamma
\in \bar\gamma+B_\delta\cap\Gamma$ .
 \end{itemize}
\end{definition}


\begin{definition}[AGDQ and QDQ on manifolds ]\label{agdqM}
Let $\mathcal{N}, \mathcal{M}$ be  $C^1$ Riemannian manifolds.  Assume that $\tilde F : \mathcal{N} \rightsquigarrow \M$       is a 
set-valued map, $(\bar \gamma,\bar y) \in \mathcal{N}\times\M  $,   $\tilde\Lambda\subset Lin\{ T_\gamma\mathcal{N}, T_y\M\} $   is a compact set,  and $\Gamma\subset\mathcal{N}$  is any  subset.
Moreover, let $\phi:U\to \R^N$ and $\psi:V\to\R^n$ be charts defined on neighbourhoods $U$ and $V$ of $\bar\gamma$ and $\bar{y}$, respectively, and assume that $\phi(\bar\gamma) = 0$, $\psi(\bar y)=0$. Consider the map 
$\psi\circ \tilde F\circ\phi^{-1}: \psi(U)\to\R^n$ and extend it arbitrarily to a map $F:\R^N\to\R^n$.
We
say that $\tilde\Lambda$ is an {\em Approximate Generalized Differential Quotient  (AGDQ) [resp. a Quasi Differential Quotient (QDQ)] of $\tilde F$ at  $(\bar \gamma,\bar y)$ } in the direction of $\tilde\Gamma$  if $\Lambda:= D\psi(\bar y) \circ \tilde \Lambda\circ D\phi^{-1}(0)$ is an  Approximate Generalized Differential Quotient  [resp. a Quasi Differential Quotient]  of $F$ at  $(0,0)$ in the direction of $\Gamma:=\phi(\tilde\Gamma\cap U)$. 
\end{definition}

As pointed out in \cite{sus},  this definition is intrinsic, that is, it is independent of the choice of the charts $\phi$ and $\psi$. 

\subsection{Open Mapping results}
\begin{thm}[Directional Open Mapping]\label{openth}
 Let $N,n$ be positive integers, and let $\Gamma$ be a convex cone in $\R^N$. Let
$F : \R^N \rightsquigarrow  \R^n$ be a set-valued map, and let $\Lambda$ be a AGDQ of $F$  at $(\bar\gamma,\bar y)$ in the direction of $\Gamma$. Let us assume that there is an element  $\bar w\in\R^n$ 
such that $\bar w\in Int(L\cdot\Gamma)$ for every $L\in\Lambda$. Then
 there exist a closed convex cone $D\subseteq\R^n$  and
positive constants $\alpha,\beta$ verifying $\bar w\in Int(D)$ and
  \bel{fixed} \bar y+\big(B_a\backslash \{0\}\cap D\big)\subset  F(\bar\gamma + B_{a\beta}\cap \Gamma) \quad \hbox{for }\; \hbox{all} \quad a\in ]0,\alpha].  \eeq
 
 \end{thm}
%
%

If one  takes $\bar w =0 $  in the statement of Theorem \ref{openth},  the cone $D$  necessarily coincides with the whole $\R^n$. As a consequence, one  obtains  the following `punctured' Open Mapping Theorem. 
\begin{cor}[`Punctured' Open Mapping]\label{bucoth}\footnote{ The adjective {\it punctured} here refers to the fact  that $\bar y$ does not belong the image $F(\bar\gamma + B_{a\beta}{\cap \Gamma})$.}
Let $N,n$ be positive integers, and let $\Gamma$ be a convex cone in $\R^N$. Let
$F : \R^N \rightsquigarrow  \R^n$ be a set-valued map, and let $\Lambda$ be an QDQ of $F$  at $(\bar\gamma,\bar y)$ in the direction of $\Gamma$. 
Let us assume that $\Lambda$ is {\rm surjective}, by which we mean that $L\cdot\Gamma = \R^n$  for every $L\in\Lambda$. Then
 there are
positive constants $\alpha,\beta$ verifying
 \bel{fixedvor} \bar y+\big(B_a\backslash \{0\}\big)\subset  F(\bar\gamma + B_{a\beta}\cap \Gamma) \quad \hbox{for }\; \hbox{all} \quad a\in ]0,\alpha].  \eeq
%
%
\end{cor}}

 If we replace AGDQ's with QDQ we get a real, non-punctured, open mapping result:
\begin{thm}[Open Mapping]\label{om}
Let $N,n$ be positive integers, and let $\Gamma$ be a convex cone in $\R^N$. Let
$F : \R^N \rightsquigarrow  \R^n$ be a set-valued map, and let $\Lambda$ be a GDQ of $F$  at $(\bar\gamma,\bar y)$ in the direction of $\Gamma$. 
As above, let us assume that $\Lambda$ is {\rm surjective}, by which we mean that $L\cdot\Gamma =  \rr^n$  for every $L\in\Lambda$. 
Then the following statements $(i), (ii)$ hold true:
 \begin{itemize}
 \item[$(i)$] 
 there are
positive constants $\alpha,\beta$  having the property that   
 \bel{fixedvorbis} \bar y+\big(B_a\backslash \{0\}\big)\subset  F(\bar\gamma + B_{a\beta}\cap \Gamma) \quad \hbox{for }\; \hbox{all} \quad a\in ]0,\alpha]; \eeq
   \item[$(ii)$] there exists  $\check\delta>0$ such that, for every $\delta\leq\check\delta$  and every    $(L_\delta,h_\delta)$ as in Definition \ref{qdq},  there exists $\gamma_\delta$  $\in\bar\gamma+\Gamma\cap B_{\delta}$ such that 

    \bel{superfix}\bar y =   L_{\delta}(\gamma_{\delta} )\cdot(\gamma_{\delta}- \bar\gamma)  + h_{\delta}(\gamma_{\delta} )\  \ \ \Big[\in F(\gamma_\delta)\Big].\eeq 
%
%
 \end{itemize} 
 In particular, by possibly reducing the size of $\alpha$, 
 one gets the open-mapping inclusions
$$\bar y + B_a \subset F(\bar\gamma + B_{a\beta}\cap \Gamma) \quad \hbox{for }\; \hbox{all} \quad a\in ]0,\alpha].$$ 


\end{thm}

\begin{proof}
 Without loss of generality, we can assume $(\bar\gamma,\bar y)=(0,0)$ and,
since a QDQ is an AGDQ, it is sufficient  to prove only statement $(ii)$.  Namely, for every ${\delta}>0$ sufficiently small, we have to establish the existence of a $\gamma_{\delta} \in  B_{\delta}\cap\Gamma$ such that 
\bel{goal}0  =   L_{\delta}(\gamma_{\delta} )\cdot\gamma_{\delta}  + h_{\delta}(\gamma_{\delta} ).\eeq


 For every $\delta>0$,  let us define the set-valued map  $L^{{-1}_r}_\delta:
 B_\delta\cap\Gamma \rightsquigarrow Lin(\R^n,\R^N)$ by setting, for every $\gamma\in  B_\delta\cap\Gamma $,
$$
L^{{-1}_r}_\delta(\gamma):= \Big\{M\in Lin(\R^n,\R^N) , \ L_\delta(\gamma)\circ M  =Id_{\R^n}  \Big\}.
 $$
 
 Namely, {\it$L^{{-1}_r}_\delta(\gamma)$ is the set of  right inverse  of $L_\delta(\gamma)$}. Let us first observe that, for every $\gamma\in  B_\delta\cap\Gamma $,   and  $\delta$ sufficiently small, $L^{{-1}_r}_\delta(\gamma)$ is non-empty. Indeed, it contains  the  {\it  Moore-Penrose pseudo-inverse} $$M_\delta^\sharp(\gamma):= L_\delta^{tr}(\gamma)\circ\left(L_\delta(\gamma)\circ L_\delta^{tr}(\gamma)\right)^{-1}, $$ where $^{tr}$ denotes transposition. 
 Furthermore,  it is trivial to verify that
{\it the set-valued map $L^{{-1}_r}_\delta$ is convex-valued. } Finally,
by possibly reducing the size of $\check\delta$, for every $\delta\in[0,\check\delta]$, {\it the set-valued map $L^{{-1}_r}_\delta$ has compact graph.} 
Indeed, there exist a constant $ K >0$ such that $\Lambda^{\rho({\delta})}$ is a compact subset  made of linear operators whose { right inverse} are bounded  (in the operator norm) by  $K$. Moreover, let us consider a sequence $(\gamma_m)_{m\in\mathbb{N}}\subset B_\delta\cap\Gamma$ converging to 
$\tilde \gamma\in B_\delta\cap\Gamma$, and, for every $m\in\mathbb{N}$, let us choose
 $M_m\in L^{{-1}_r}_\delta(\gamma_m)$. Hence, one has that
{ $L_\delta(\gamma_m) \circ M_m = Id_{\R^n}$}  and, since the sequence $(M_m)$ ranges in a compact set, there exists a subsequence  $(M_{m_k})$ converging to a linear operator $\tilde M$.  In particular, 

$$   L_\delta(\tilde \gamma)\circ \tilde M=\ds\lim_{k\to\infty} \big( L_\delta(\gamma_{m_k})\circ M_{m_k} \big) = Id_{\R^n},$$ 
so that  $
\tilde M\in L^{{-1}_r}_\delta(\tilde\gamma)$. This proves that the set-valued map
 $\gamma\mapsto L^{{-1}_r}_\delta(\gamma)$ has compact graph. 

Now consider the set-valued map $\Psi_\delta:
 B_\delta\cap \Gamma \rightsquigarrow \R^N$ defined by setting
$$
\Psi_\delta(\gamma):=\Big\{-M\cdot  h_\delta(\gamma)\ |\quad M\in L^{{-1}_r}_\delta(\gamma) \Big\}{\bigcap \Gamma}, \quad   \gamma\in B_\delta\cap \Gamma.
$$
To prove that this map has non-empty values for every $\gamma\in B_\delta\cap \Gamma$, it is sufficient  to determine a linear mapping  $M^\flat:\rr^n\to\rr^N$ and an element $v\in \Gamma$ such that 
\bel{mstar}
(L_\delta(\gamma)\circ M^\flat)\cdot w = w\quad \forall w\in\rr^n \ \ \Big(\iff  M^\flat\in  L^{{-1}_r}_\delta(\gamma)\Big) , \qquad -M^\flat\cdot h_\delta (\gamma) =v 
\eeq Fix $\gamma\in  B_\delta\cap \Gamma$ and choose $v\in\Gamma$  verifying  $L_\delta(\gamma)\cdot v=- h_\delta (\gamma)$. Such a $v$ exists, since   $L_\delta(\gamma)$ is surjective. 
Now, a geometrical intuition suggests that $M^\flat$ might be obtained by adding a suitable linear operator to an element of $L^{{-1}_r}_\delta(\gamma) $, for instance the linear operator $M^\sharp$. Actually, following \cite{sus}, if $\langle \cdot, \cdot\rangle$ is any scalar product  on $\rr^n$,
we  define the linear map $M^\flat:\rr^n\to\rr^N$ by setting, for every $w\in\R^n$,
$$M^\flat\cdot w:= M^\sharp\cdot w - \frac{\langle w, h_\delta (\gamma)\rangle}{ \langle  h_\delta (\gamma), h_\delta (\gamma)\rangle} \left( v - M^\sharp\cdot  h_\delta (\gamma)\right).
$$
It is straightforward to verify that $M^\flat$ verifies conditions \eqref{mstar}, so that $\Psi_\delta(\gamma)$ is not empty.


 Since for every $\delta$ the map $h_\delta$ is continuous and $\|h_\delta\|\leq \delta\rho(\delta)$, by possibly reducing the size of
$\check\delta$ we conclude   that, for every $\delta\in [0,\check\delta]$, the set-valued map $\Psi_\delta$ verifies $\Psi_{\delta}( B_{\delta}\cap \Gamma)\subset \bar B_{\delta}{\cap \Gamma}
$ and has non-empty, convex values, and  a closed graph. Since the domain of $\Psi_{\delta}$ is compact and convex,  the set-valued map $\Psi_{\delta}$ verifies the hypotheses of the Kakutani fixed point theorem, so that there exists $\gamma_{\delta}\in B_\delta\cap \Gamma$ such that 
$
\gamma_{\delta} \in   \Psi_{\delta}(\gamma_{\delta}).
$
It follows that there is a matrix { $M\in L^{{-1}_r}_\delta(\gamma_{\delta})$ 
 such that  $0= \gamma_{\delta} +M\cdot h_{\delta}(\gamma_{\delta}).$} Therefore, one gets 
 $$0= L _{\delta}(\gamma_{\delta})\cdot \Big(\gamma_{\delta} +M\cdot h_{\delta}(\gamma_{\delta})\Big)= L _{\delta}(\gamma_{\delta})\cdot \gamma_{\delta} +h_{\delta}(\gamma_{\delta}),$$
which concludes the proof. 

\end{proof}

\subsection{QDQ approximating cones and set separation}

Assume that $\M$ is a $C^1$ differentiable manifold , $\mathcal E\subset\M$, and $z\in \mathcal E$. If $X$ is a linear space, let us call {\it convex multicone} in  $X$ any family of convex cones of $X$.

\begin{definition}\label{ApprCone} An {\em  AGDQ approximating multicone  [}resp. {a
\em  QDQ approximating multicone]} to $\mathcal E$ at $z$ is a
convex multicone ${\mathcal C}\subseteq T_{z}\M$ such that there   exist a non-negative integer   $N$, a set-valued
map $F : \R^N  \rightsquigarrow \M$, a convex cone $\Gamma\subset\R^N$, and an AGDQ  {\em[}resp. a {\em QDQ] }$\Lambda$ of F at  $(0,z)$ in the direction of $\Gamma$
such that $F(\Gamma)\subset \mathcal E$ and ${\mathcal C} =
\{L\cdot\Gamma\  : L\in\Lambda\}$.

In  the particular case when an AGDQ approximating multicone   {\em  [ resp. a
QDQ approximating multicone]} is a singleton, namely $\Lambda=\{L\}$ for some $L\in Lin(\R^N,\R^n)$, we say  that $C:=L\cdot\Gamma$ is   an {\em AGDQ approximating cone}   {\em [ resp. a
  QDQ approximating cone]}  to $\mathcal E$ at $z$.

\end{definition}

Let us introduce the notion of local set-separation:
\begin{definition} Let  $\mathcal{X}$ be a topological space, and let us consider two subsets ${\mathcal A}_1,{ \mathcal A}_2\subset \mathcal{X} $ and a point   $z\in  {\mathcal A}_1\cap {\mathcal A}_2$. We say that ${\mathcal A}_1$ and ${ \mathcal A}_2$ are {\em  locally separated} at $z$ provided there exists a neighborhood $V$ of $z$ such that
$$
{\mathcal A}_1\cap{ \mathcal A}_2\cap V = \{z\}.
$$
\end{definition}
\vskip0.4truecm

We   are  now  ready to state  our  set-separation result, which connects set separation with the linear separability of 
QDQ approximating cones. Furthermore the result includes a special property in the case when the approximating cones are complementary linear subspaces.

\begin{thm}[Set separation]\label{teoteo}
Let  $ {\mathcal E}_1, {\mathcal E}_2 $ be subsets of  $\N $, and 
  let
  $ z\in {\mathcal E}_1\cap {\mathcal E}_2 $
 Assume that  $ C_1$, $C_2$ are   AGDQ approximating cones  of  $ {\mathcal E}_1 $ and $ {\mathcal E}_2 $, respectively, at $z$. 
 \begin{itemize}
 \item[i)] If  $ C_1$ and $ C_2$ are strongly transverse,
then  the sets ${\mathcal E}_1$ and ${\mathcal E}_2$
are not locally separated. 
\item[ii)] If moreover:   \begin{enumerate}
\item $ C_1$, $C_2$ are  QDQ cones,
 \item $C_1$ and $C_2$ are complementary  linear subspaces, i.e.    ${C}_1\oplus{C}_2=T_{z}\M$,
  \item for each $i=1,2$,  $\Gamma_i\subset\R^{N_i}$ is a convex cone, $F_i:\R^{N_i}\rightsquigarrow \M$ is a set-valued map,   and  $\Lambda_i=\{L^i\}\in Lin(\R^{N_i},T_{z}\M)$  is a 
QDQ of $F_i$ at $(0,z)$ in the direction of $\Gamma_i$, $F_i(\Gamma_i)\subseteq \mathcal{E}_i$, and $ C_i = L^i\cdot \Gamma_i$, 

 \end{enumerate}
   then there exists a sequence $(\gamma_{1_k},\gamma_{2_k})\in \Gamma_1\times\Gamma_2$
 such that $z_k\in F_1(\gamma_{1_k})\cap  F_2(\gamma_{2_k})$ and 
   $z_k\to z$.

   
   

\end{itemize}
\end{thm}


\begin{remark}{\rm Property $ii)$, whose proof is based on the Open Mapping result stated in Theorem \ref{om}, is not true if we replace QDQ approximating cones  with AGDQ approximating cones. Of course, this is connected with the non validity of a non-punctured open mapping result for AGDQ's.

}\end{remark}
\begin{proof}[Proof of Theorem \ref{teoteo}]
{ Statement $i)$ of Theorem \ref{teoteo} is direct consequence of  \cite{sus}, Theorem 4.37, where an analogous   result concerning the non-separation of multicones is provided 

Let us prove  statement $ii)$. 
Because of the local character of the statement, there is not loss of generality in considering only the Euclidean case
when $\mathcal{M}=\R^n$. For every $i=1,2$,  let  $n_i\geq 0$   be the dimensions of  the subspace $C_i$ (so that $n_1+n_2=n$), and let  $N_i\geq 0$ an integer such that  $\Gamma_i\subset\R^{N_i}$. 
By hypothesis, for every $i=1,2$
there exists
a modulus  $\rho_i:[0,+\infty[\,\rightarrow [0,+\infty[$
having the property that,
 for every $\delta>0$,  there exists  a continuous map  $(L_{\delta}^i,h_\delta^i):B_{\delta}\cap\Gamma_i\to Lin\{ \R^{N_i}, \R^{n_i}\} \times \R^{n_i}$, 
   such that 

 $$ |L^i_\delta(\gamma_i) - L^i|\leq \rho_i(\delta),
 \quad |h_\delta^i |\leq \delta \cdot \rho_i(\delta), \quad\hbox{and}\
  z+  L_\delta^i(\gamma_i)\cdot \gamma_i   + h_\delta^i(\gamma_i)\in F_i(\gamma_i)
 $$
whenever $\gamma_i\in B_\delta\cap\Gamma_i$.
Let us consider the cone $\Gamma:=\Gamma_1\times \Gamma_2\subset\R^{N_1+N_2} $ and the set-valued map $F:\Gamma \rightsquigarrow \R^{n}$ defined by setting
$$
F(\gamma_1,\gamma_2) := F_2(\gamma_2) -  F_1(\gamma_1) = \Big\{z_2-z_1  \  | \   (z_1,z_2)\in F_1(\gamma_1)\times F_1(\gamma_2) \Big\} \quad \forall
(\gamma_1,\gamma_2)\in \Gamma_{1}\times\Gamma_2,  
$$
and observe that  $$\hbox{\it if}\ (\bar\gamma_1,\bar\gamma_2) \ \hbox{\it is such that} \ 0\in F(\bar\gamma_1,\bar\gamma_2) \;\hbox{\it then} \; \  \emptyset \neq F_2(\bar\gamma_2) \cap  F_1(\bar\gamma_1)\subseteq \mathcal{E}_2\cap \mathcal{E}_1.$$ 
 Furthermore, let us set $\rho(\delta):= \rho_1(\delta)+\rho_2(\delta)$ and let us define the continuous map

$$(L_\delta,h_\delta)(\gamma_1,\gamma_2):= \Big(\left(L^2_\delta(\gamma_2), -  L^1_\delta(\gamma_1)\right)\ ,\  h^2_\delta(\gamma_2)- h^1_\delta(\gamma_1) \Big), \qquad(\gamma_1,\gamma_2)\in B_{\delta}\cap\Gamma. $$

Defining the linear map $L\in Lin\{\R^{N_1+N_2},\R^n\}$ by  setting  $L(v_1,v_2):= L^2\cdot v_2-L^1\cdot v_1$, one obtains
 $
|L_\delta(\gamma_1,\gamma_2) - L|\leq \rho(\delta)$,  $|h_\delta |\leq \delta \cdot \rho(\delta)$, and
$$
   L_\delta(\gamma_1,\gamma_2)\cdot (\gamma_1,\gamma_2)  + h_\delta(\gamma_1,\gamma_2)\in F(\gamma_1,\gamma_2).
$$
whenever $(\gamma_1,\gamma_2)\in B_\delta\cap\Gamma$. 
Hence, $\Lambda$ is a QDQ of $F$ at $(0,0)$. Moreover, one has $L\cdot\Gamma= C_1\oplus C_2 =\R^n$  so that, by the Open Mapping result stated in Theorem \ref{om} for $k\in\mathbb{N}$ sufficiently large,  we get the existence of 
$(\bar\gamma_1^{\frac1k},\bar\gamma_2^{\frac1k})\in \Gamma\cap  B_{\frac{1}{k}} \subset \Gamma_1\times \Gamma_2$
such that, setting $(\gamma_{1_k},\gamma_{2_k}):=(\bar \gamma_1^{\frac1k}, \bar \gamma_2^{\frac2k})$, one has
$$z_{k}:= z+  L^1_{\frac1k}(\gamma_{1_k})\cdot \gamma_{1_k} + h_{\frac1k}^1(\gamma_{1_k}) =  z +  L^2_{\frac1k}(\gamma_{2_k})\cdot \gamma_{2_k} + h_{\frac1k}^2(\gamma_{2_k})  \in F_1(\gamma_{1_k})\cap  F_2(\gamma_{2_k}).$$
Notice that,  by $ h_{\frac1k}^1(\gamma_{1_k})\leq \rho_1({\frac1k})\cdot{\frac1k}$, and $|\gamma_{1_k}|\leq {\frac1k}$ one has   $\ds\lim_{k\to \infty}z_k = z$, which concludes the proof. }
\end{proof}

\section{Gaps and set-separation}

\subsection{Original and extended controls}
 Let $({\mathcal M},{\langle\cdot,\cdot\rangle})$ be  a  Riemannian  $C^{2}$-differentiable manifold, let $ [0,S]$ be a {\it time}-interval and let  $\FW$ be a metric space which we call the set  of {\it control values}.  For every $(s,\w)\in [0,S]\times\FW$, let  $\mathcal{M}\ni y\mapsto (y, f(s,y,\w))\in T\mathcal{M}$  be a vector field. 
We will consider {\it two families of controls $\V$, $\W:= L^1([0,S],\FW)$}, with 
$  \V\subset\W.$
We will call 
$\V$ and $\W$   the {\it original  family of controls} and the
 {\it extended family of controls}, respectively.

Let us choose an {\it initial point} $\bar y\in\mathcal{M}$, and, for any {\it  control} map  $w\in\W$, let us consider the Cauchy problem
\[
(E)\,\qquad\qquad
\left\{\begin{array}{l}\displaystyle \ds\frac{dy}{ds}(s)={f}(s, y(s),w(s))\quad\mathrm{a.e.\,\, s\in [0,S]}\\[5mm]
y(0)=\bar y
\end{array}\right.  .
\]
\vskip5truemm

We shall assume the following regularity hypothesis:

 {\bf Hypothesis (SH) :}
\begin{itemize}
\item[(i)] for each $(s,\w)\in [0,S]\times \mathfrak{W}$, the vector field  $y\mapsto{f}(s,y,\w)$  is of class $C^1$ on $\M$;
\item[(ii)] there exists  an integrable function $c\in L^{1}([0,S]; \R)$
such that,  for a.e. $s\in [0,S]$, 
\begin{equation}\label{freg}
\left|{f}(s,y,\w) \right|\leq c(s),\qquad 
\left|D{f}(s,y,\w)\right| \leq c(s)
\end{equation}
for every $(y,\w)\in  \M\times \FW$.
\item[(iii)] for every $(y,\w)\in \mathcal{M} \times \FW$,   the map $s\mapsto {f}(s, y,\w)$ is measurable;
\item[(iv)]  for every  $s\in [0,S]$, the map $(y,\w)\mapsto {f}(s, y,\w)$ is continuous.
\end{itemize}

  In particular, for every $w\in\W$ there exists a unique trajectory $y[w]$ of (E).



%
%
%

\vskip3truemm
Let us fix a closed set  $\cruk\subseteq {\mathcal M}$, which we will refer to as  
  {\it target}.
  
\begin{remark}{\rm   Of course, through standard cut-off arguments, in many situations one can  replace (ii) in hypothesis (SH) with a weaker assumption concerning a neighbourhood of $\hat y([0,S])$ istead of the whole state-space $\M$.} 
\end{remark}
 
 \begin{definition} For any control $v\in\V$ {\rm[}resp. $w\in\W${\rm ]},
   the pair $(y, v):=(y[v],v)$ {\rm[}resp. $(y, w):=(y[w],w)${\rm ]} will be  called {\rm original process}  {\rm[}resp. {\rm extended processes}{\rm ]}. 
 An extended process  --in particular, an original process-- $(y, w)$ is called {\rm feasible} if $y(S)\in \T$. 
  \end{definition}

\subsection{Infimum gaps}

Let us endow the set of controls $\W$ with the pseudo-distance $d_{f}$ defined as 
 \begin{equation}\label{distanza}
 d_{f} (w_1,w_2) := d_\infty(y[w_1],y[w_2])\;\Big(:= \max_{s\in [0,S]} d(y[w_1](s),y[w_2](s))\Big),
 \end{equation}
  for all controls $w_1, w_2 \in \W$. 
  

The set 
\begin{equation}\label{original_reach_set}
\mathcal{R}_\V:= \Big\{y[v](S):\quad v\in \V \Big\} \subset \N
\end{equation}
 will be  called the  \textit{original reachable set},
 and the set 
\begin{equation}\label{ext_reach_set}
 \mathcal{R}_\W:= \Big\{y[w](S):\quad w\in \W \Big\}   \subset \N
\end{equation}
will be  called the \textit{extended reachable set}.

We will also consider  local  versions of the above reachable sets. Precisely, for a given extended process  $(\hat y,\hat w)$ and $r\geq 0$, we set
\[
\begin{array}{l}\mathcal{R}^{\hat w,r}_\V:= \Big\{ y[v](S)\,:\qquad  v\in\V,\quad  d_{{f}}( \hat w, v)<r\ \Big\}
\\ \mathcal{R}_{\W}^{\hat w,r} := \Big\{ y[w](S)\,:\qquad  w\in\W,\quad  d_{{f}}( \hat w, w)<r  \Big\}
.\end{array}\]
Clearly $\mathcal{R}_\W\supseteq \mathcal{R}_\V$ and $\mathcal{R}_\W^{\hat w,r}\supseteq \mathcal{R}_\V^{\hat w,r}$, for all $r\geq 0$.


\vskip0.4truecm

 The occurrence  of  a local infimum gap is captured by the following definition:
\begin{definition}\label{gap-generator2}  Let $(\hat{y}, \hat{w})$ be a   feasible extended process such that  $\hat{y}(S)\in \mathcal{R}_{\W}\backslash \mathcal{R_\V}$.  
We say that {\rm $(\hat{y}, \hat{w})$
   satisfies the infimum  gap condition}  if, for any continuous {\it cost}  function $h:\mathbb{R}^{n}\rightarrow \mathbb{R}$, there exists $r>0$ such that one has 
\bel{igc}
h\big(\hat{y}(S)\big) <\inf\Big\{h(y)\,:\quad y\in\mathcal{R}^{\hat w,r}_{\V} \cap \cruk \Big\}
\eeq

\end{definition}

 Despite the name, the infimum  gap condition \eqref{igc} is clearly a fully { dynamical} property. Actually, it could be as well rephrased in terms of  `supremum gap' or even independently of any optimization procedure as shown in Lemma \ref{lemma infgap} below.

\begin{definition}  Let $(\hat{y}, \hat{w})$ be a   feasible extended process such that  $\hat{y}(S)\in \mathcal{R}_{\W}\backslash \mathcal{R_\V}$.  We say that 
 {\rm $(\hat{y}, \hat{w})$ is  isolated from $\V$} if,
for some $r>0$
the sets  $\Big(\mathcal{R}_\V^{\hat w,r}\cup \{ \hat{y}(S)\}\Big)$ and $\cruk$ are locally separated at $\hat{y}(S)$, namely, there exists a neighborhood $\mathcal{N}\subset \mathcal{M}$ of $\hat{y}(S)$ such that $\Big(\mathcal{R}_{\V}^{\hat w,r}\cup \{\hat{y}(S) \}\Big)\cap\cruk\cap \mathcal{N} =\{ \hat{y}(S)\}$.
\end{definition}

 \begin{lma}\label{lemma infgap} Let $(\hat y,\hat w)$ be an extended feasible  process   such that $\hat{y}(S)\in \mathcal{R}_{\W}\backslash \mathcal{R_\V}$. 
Then the following conditions are equivalent:
\begin{itemize}
\item[i)]  $(\hat y,\hat w)$ satisfies \eqref{igc} for a given continuous cost function $h$  and $ \hat{r}>0$;

\item[ii)] 
the process $(\hat{y}, \hat{w})$ is isolated from $\V$;

\item[iii)]  the process  $(\hat y,\hat w)$ satisfies  the infimum gap condition. Furthermore the right hand-side  of \eqref{igc} is  equal to $+\infty$.

 \end{itemize}
\end{lma}

\begin{proof}   We give a proof just for the sake of  completeness, all arguments being trivial.

  Let us start proving that $i)$ implies $ii)$. This means that 
one has to show that  there exists $\hat r>0$ such that
\bel{tisolated}
\mathcal{R}^{\hat w,r}_{\V} \cap \cruk =\emptyset \quad \forall r<\hat r. \eeq  Assume that \eqref{tisolated} is false, which means that there exists a sequence  $r_{n}\downarrow 0$ such that $\mathcal{R}^{\hat w,r_n}_{\V} \cap \cruk \neq \emptyset$ for all natural $n$. This implies that there exists a sequence $(y_{k})_{k\in\mathbb{N}}$ verifying  $y_{k}\in \left( \mathcal{R}^{\hat w,r_k}_{\V} \cap \cruk \right)$ for every $k\in \mathbb{N}$, so that $y_n\rightarrow \hat{y}(S)$, which, in view of the continuity of $h$, contradicts $i$).   Hence, \eqref{tisolated}  holds true, from which  we get $ii$). 

Let us now prove that $ii) \Rightarrow iii)$. By hypothesis,  
there exists a neighborhood $\mathcal{N}$ of $\hat{y}(S)$ such that $\Big(\mathcal{R}_{\V}^{\hat w,r}\cup \{\hat{y}(S) \}\Big)\cap\cruk\cap \mathcal{N} =\{ \hat{y}(S)\}$. Since  $\hat{y}(S)\in \mathcal{R}_{\W}\backslash \mathcal{R_\V}$,  by possibly reducing the size of $r>0$ one obtains that   $\mathcal{R}^{\hat w,r}_{\V} \cap \cruk =\emptyset$,  which obviously implies  $iii)$, with the right hand-side  of \eqref{igc}   equal to $+\infty$. Finally, the relation  $iii)\Rightarrow i)$ is trivial. 

\end{proof} 

\section{Abundance}
 
Our main results  --namely Theorem \ref{GEOM} and Corollaries \ref{MP_INF_GAP} and  \ref{mainth}-- strongly rely on a property  introduced  by J. Warga and called ``abundance''. It consists in  a particular  pervasiveness of $\V$ in $\W$, which  happens to be stronger than density. In fact, because of the presence of a closed final constraint, the mere  density of  $\mathcal{R}_{\V}$ into $\mathcal{R}_{\W}$ is not enough in order to normality to be   a sufficient condition  for gaps' avoidance (see  Subsection \eqref{ex_suss}).  We will make use of  a  generalization of abundance  provided in  \cite{Kaskosz} and we will extend  it to manifolds.

 For every positive integer $N$, let $\Gamma_{N}$ be  the convex hull of the union of the origin with the $N$-simplex, namely 
$$\Gamma_{N}:=\Big\{\gamma=({\gamma}^1,...,{\gamma}^{N})\in \mathbb{R}^{N}: \sum_{j=1}^{N}{\gamma}^{j}\leq 1,\; {\gamma}^{j}\geq 0, \; j=1,...,N\Big\}.$$

For any $\gamma\in \Gamma_{N}$, let us consider the control system on $\M$
\begin{equation}\label{ab_eqF}
\left\{\begin{array}{l}
\ds \frac{dy}{ds}(s)= {f}_\gamma \Big(s,y(s),w(s), w_{1}(s),...,w_{N}(s)\Big)\\\,\\y(0)=\bar y,\end{array}\right.
\end{equation}
where:
 i) the {\it control}  $(w, w_{1},...,w_{N})$ belongs to $\W^{1+N}$,  and
 ii) 
 the  vector field ${f}_{\gamma}$ is defined by setting, for every
 $(s,y)\in[0,S]\times {\mathcal M}$  and $(\w,\w_{1},...,\w_N)\in\FW^{1+N}$,
$$
{f}_{\gamma} \big(s,y,\w, (\w_{1},...,\w_N)\big) :=
{f}(s,y, \w) + \sum_{i=1}^{N}{ \gamma}^{i}\Big({f}(s,y, \w_{i}) - {f}(s,y, \w)\Big).
$$


For every value of the parameter  $\gamma\in\Gamma_{N}$ and every control $\big(w, w_{1},...,w_{N}\big)\in\W^{1+N}$, let us use $y_\gamma\Big[{w, w_{1},...,w_{N}}\Big]$ to denote the corresponding  solution of \eqref{ab_eqF}.\footnote{ Under hipothesis (SH) such a solution exists and is unique.} Notice, in particular,   that  $y[w] =
 y_{\gamma}\Big[{w, w,...,w}\Big] $ for all $w\in \W$ and for all $\gamma\in \Gamma_{N}.$

\begin{definition}\label{abundant}\cite{Kaskosz}
We say that a subclass of controls $\mathcal{V}\subset \mathcal{W}$ is
 {\rm abundant in $\mathcal{W}$} if,  for every integer $N$, every $(1+N)$-tuple of controls $(w, w_{1},...,w_{N}) \in \mathcal{W}^{1+N}$, and  every $\delta>0$, there exists a continuous
  mapping $\theta^{\delta}_{w, w_{1},...,w_{N}}: \Gamma_{N}\rightarrow \mathcal{V}$ such that 
 \begin{equation}\label{stimab} d\Big(y_{\gamma}\Big[w, w_{1},...,w_{N}\Big](S), y\big[\theta^{\delta}_{w, w_{1},...,w_{N}}(\gamma)\big](S)\Big)<\delta,\qquad \forall\gamma \in \Gamma_{N}.\quad
 \end{equation}

\end{definition}

\vskip5truemm
A sufficient condition for abundance, based on  {\it concatenation}, is given in Proposition \ref{Kaskosz} below.

 \begin{definition}\label{concat}
We say that a set of controls $\mathcal{V}$ satisfies the {\rm concatenation
property} if, for every $\bar s\in ]0,S[$  and for any $v_{1},v_{2}\in \mathcal{V}$, one has 
$ v_{1}\chi_{[0,\bar s[}+v_{2}\chi_{[\bar s, S]}\in \mathcal{V},\footnote{Concatenation is weaker than  {\it decomposability} of a set $\mathcal{S} $ of paths on  an interval $[a, b]$  \cite{F,Ma,O},  which prescribes  that for any pair of paths $v_{1},v_{2}\in\mathcal{S}
$ and any measurable set $E\subset [0,S]$, one has 
$$
 v_{1}\chi_{E}+v_{2}\chi_{([0,S]\setminus E)}\in \mathcal{S}.
$$}
$
where we have used $\chi_{E}$ to denote the  indicator function of  a subset  $E\subseteq [0,S]$

\end{definition}

\begin{prop}\label{Kaskosz}\cite{Kaskosz}
Assume that the subfamily $\mathcal{V}\subset \mathcal{W}$ satisfies the concatenation property and is dense in $\mathcal{W}$  with respect to the pseudo-metric $d_{f}$.
Then $\mathcal{V}$ is an abundant subset of $\mathcal{W}$. 
\end{prop}

 The proof of this result  for the  special case  when $\mathcal{M}=\R^{n}$ was given in (\cite{Kaskosz}, Theorem IV.3.9)  by developing some arguments in \cite{Ga}.
 The required, obvious,  changes to prove the result  on a Riemannian manifold reduce to a reformulation of estimate \eqref{stimab} in local coordinates, so we omit them.

 \subsection{Approximating  the original reachable set by extended cones }
 Let  us fix a  a feasible extended process $(\hat y,\hat w)$, and, for  any $s,\check s\in[0,S]$, $s>\check s$, let $M(s,\check s):T_{\hat y(\check s)}\mathcal{M} \to T_{\hat y(s)}\mathcal{M}$  denote the differential of the  diffeomorphism established by the differential equation 
 $\dot y=f(s,y,\hat w)$  from a neighborhood of $\hat y(\check s)$  to a neighborhood of $\hat y(s)$.
 As is known, $s\to M(s,\check s)$  is the solution of the {\it variational} Cauchy problem having the following  coordinate representation:
 
\begin{equation}\label{linearized_problem}
  \frac{d {M}}{d s} (s)= \frac{\partial {f}}{\partial y}(s,\hat{y}(s),\hat{w}(s))\circ{M}(s), \qquad M(\check s,\check s) = id_{T_{\hat y(\check s)}\N}.
\end{equation}

\begin{definition}  
Consider a  positive  integer $N$, $N$ control values  $\mathfrak{w}_{1},...,\mathfrak{w}_{N} \in \mathfrak{W}$, and $N$ instants
  $ s_{1},...,s_{N}\in \mathrm{SD}\{{f}(\cdot,\cdot, \hat{w}(\cdot))\}\cap 
\mathrm{SD}\{{f}(\cdot,\cdot,\w_1)\}\cap\ldots\cap 
\mathrm{SD}\{{f}(\cdot,\cdot,\w_N)\}$,\footnote{We have used the notation $SD(\phi)(\cdot,\cdot)$  --introduced in Subsection \ref{basicss}--  to mean the (full measure) Scorza-Dragoni set of a function $\phi=\phi(s,y)$}
$0< s_1<\ldots,< s_N\leq S$ . 
The  convex cone\begin{small}
$$
{\bf C}_{\mathfrak{w}_{1},...,\mathfrak{w}_{N} }^{s_{1},...,s_{N}} =
 span^+\left\{M(S,s_{i})\cdot \Big( {f}(s_{i},\hat{y}(s_{i}), \mathfrak{w}_{i})-{f}(s_{i},\hat{y}(s_{i}), \hat{w}(s_{i})) \Big):\ i=1,\ldots,N\right\}\subset  T_{\hat y(S)} \mathcal{M}
$$\end{small}
will be called  {\rm extended variational cone}   corresponding to the feasible extended process $(\hat y,\hat w)$.

\end{definition}
 
%
%
%

 The following result can be regarded as claiming a sort of  {\it infinitesimal thickness} of  $\V$ in $\W$.

\begin{thm}\label{apprpicc} Let the  original family of controls $\mathcal{V}\subset \mathcal{W}$ be abundant in  $\mathcal{W}$, and 
let  a feasible extended process $(\hat y,\hat w)$ be given. Consider a  positive  integer $N$, $N$ control values  $\mathfrak{w}_{1},...,\mathfrak{w}_{N} \in \mathfrak{W}$, and $N$ instants
  $ s_{1},...,s_{N}\in \mathrm{SD}\{{f}(\cdot,\cdot, \hat{w}(\cdot))\}\cap 
\mathrm{SD}\{{f}(\cdot,\cdot,\w_1)\}\cap\ldots\cap 
\mathrm{SD}\{{f}(\cdot,\cdot,\w_N)\}$,
$0< s_1<\ldots,< s_N\leq S$ . 
Then, for any $r>0$,  the extended variational cones   ${\bf C}_{\mathfrak{w}_{1},...,\mathfrak{w}_{N} }^{s_{1},...,s_{N}} $ is a   QDQ approximating cone to    
 $ \mathcal{R}_\V^{\hat w,r}\cup\{ \hat{y}(S)\}$ at $\hat{y}(S)$.

\end{thm}

\begin{remark}{\em While the fact  that  ${\bf C}_{\mathfrak{w}_{1},...,\mathfrak{w}_{N} }^{s_{1},...,s_{N}}$  is  a  QDQ approximating cone to the extended reachable set $\mathcal{R}_\W^{\hat w,r}$ at $\hat{y}(S)$ (for any $r>0$) is a classical argument, utilized in the proof of the Maximum Principle,\footnote{Actually the same is true for other, more classical, cones, e.g. the Boltyanski cone and the regular tangent cone. } the fact that ${\bf C}_{\mathfrak{w}_{1},...,\mathfrak{w}_{N} }^{s_{1},...,s_{N}}$ is a first order approximation for the {\it small} reachable set $\mathcal{R}_\V^{\hat w,r}$ is anything but obvious: it means, in a sense,  that this cone is not too large.}
\end{remark}

\begin{proof}[Proof of Theorem \ref{apprpicc}] 

 We will prove this theorem assuming that $\M$ is an open subset of $\rr^n$, so that we can identify $T_{\hat{y}(S)}\M$ with $\rr^n$.
 Clearly, this is not restrictive because of the local character of the  result.
 
Let  ${\bf C}_{\mathfrak{w}_{1},...,\mathfrak{w}_{N} }^{s_{1},...,s_{N}}$ be an extended variational cone.
Let us set $s_0=0$ and, for every $ i=1,...,N$,  consider  a number $\delta_i\leq s_i-s_{i-1}$ and the  control
\begin{equation}\label{needle}
w^{\delta^i}_{i}(s):=\left\{\begin{array}{l} \hat{w}(s),\quad \forall s\in [0,S]\backslash [s_i - \delta^i, s_i]\\
\mathfrak{w}_{i}\qquad \forall s\in [s_i - \delta^i, s_i].
\end{array}\right. \qquad
\end{equation}
Let us set
 $\bar\delta :=\frac{N}{2} \min\{s_i-s_{i-1} , \ i=1,\ldots,N\}$.
Let us define the set-valued map $F:\R^N\rightsquigarrow \R^n
 $  as
\begin{equation}\label{F(eps)}
F(\epsilon)=\left\{ y\Big[\theta^{\delta^2}_{\hat{w},w_{1}^{\delta/N},\ldots, w^{\delta/N}_N}\left(\pi\left(\ds\frac{\epsilon}{\delta}\right)\right)\Big](S):\quad\quad  0< \delta \leq \bar \delta  \right\}\quad \forall\epsilon\in \R^N,
\end{equation}
where
$\pi:\R^N\to \Gamma_{N}$ denotes the orthogonal projection on $\Gamma_{N}$ (which, because of the convexity of $\Gamma_{N}$, is a continuous, single-valued, map). Notice that, by construction $F(\epsilon) \subseteq   \mathcal{R}_{\V}$, for every $\epsilon\in\R^N$.

 For each $\delta\in ]0, \bar{\delta}]$ and $\epsilon\in \Gamma_N\cap  B_{\delta}$, let us choose $\gamma=\left(\gamma^1,\ldots, \gamma^N \right):=\left(\frac{\epsilon^1}{\delta},\ldots,\frac{\epsilon^N}{\delta}\right)\in \Gamma_N$.
  From \eqref{mezzavia} in Lemma \ref{claim} below it follows  that
\begin{equation}\begin{array}{l}
y_{\gamma}[\hat{w},w_{1}^{ \delta/N},\ldots, w^{ \delta/N}_N](S) = y_{\epsilon/\delta}[\hat{w},w_{1}^{\delta/N},\ldots, w^{\delta/N}_N](S)= \\\\\ds \qquad=\hat{y}(S)+\frac{1}{N}\sum_{i=1}^{N}\epsilon^i\,  M(S,s_i)\cdot \Big( f(s_{i}, \hat{y}(s_i), \w_i)- f(s_{i}, \hat{y}(s_i), \hat{w}(s_i)\Big) +\phi(\epsilon,\delta),\end{array}
\end{equation}
for every   $\epsilon\in \Gamma_N\cap  B_{\delta}$, where  $\phi:\ds\bigcup_{0<\delta\leq \bar\delta}\left((\Gamma_N\cap B_\delta)\times \{\delta\}\right)\to \R^n$     is a continuous function  which verifies
$\max\big\{|\phi(\epsilon,\delta)|, \; \epsilon\in \Gamma_N\cap B_\delta\big\} = o(\delta).  $    
In view of the abundance property, for each $\delta\in]0,\bar \delta]$ and $\epsilon\in \Gamma_N\cap B_{\delta}$, there exists $\tilde\phi:\ds\bigcup_{0<\delta\leq \bar\delta}\left((\Gamma_N\cap B_\delta)\times \{\delta\}\right)\to \R^n$ such that
$$
y\Big[\theta^{\delta^2}_{\hat{w},w_{1}^{\delta/N},\ldots, w^{\delta/N}_N}\left(\frac{\epsilon}{\delta}\right)\Big](S) -y_{\epsilon/\delta}[\hat{w},w_{1}^{\delta/N},\ldots, w^{\delta/N}_N](S)= \tilde{\phi}(\epsilon,\delta)
$$
for all $\epsilon\in\Gamma_N\cap B_\delta$, with   $\max\big\{|\tilde{\phi}(\epsilon,\delta)|, \;\epsilon\in \Gamma_N\cap B_\delta\big\} \leq \delta^2$.
Therefore 

\begin{equation}\label{ab}\begin{array}{l}
\qquad y\Big[\theta^{\delta^2}_{\hat{w},w_{1}^{\delta/N},\ldots, w^{\delta/N}_N}\left(\ds\frac{\epsilon}{\delta}\right)\Big](S)=\\ \qquad =\hat{y}(S)+ \ds\frac1N \sum_{i=1}^{N}\epsilon^i\, M(S,s_i)\cdot \Big( f(s_{i}, \hat{y}(s_i), \w_i)- f(s_{i}, \hat{y}(s_i), \hat{w}(s_i)\Big) +h^\delta(\epsilon),\end{array}
\end{equation}
where 
$h^\delta(\epsilon) :=\phi(\epsilon,\delta) + \tilde{\phi}(\epsilon,\delta)$.
Observe that 
\bel{h}
|h^\delta(\epsilon)|\leq \delta \rho(\delta),\quad\forall\epsilon\in B_\delta,
\eeq
where we have set 
$\ds
\rho(\delta) := \frac{\ds\max\big\{|\phi(\epsilon,\delta)|, \; \epsilon\in \Gamma_N\cap B_\delta \big\}}{\delta} + \delta.
$

For every $\delta\in]0,\bar\delta]$, let  us   define  the map $$
\begin{array}{l}
A^{\delta}:\Gamma_N\cap B_\delta \to Lin(\R^N,\R^n)\times \R^n\\
\qquad\qquad \ \ \epsilon \mapsto A^{\delta}(\epsilon):=(L, h^{\delta}(\epsilon)),
\end{array}$$ where $L$ is the linear map defined as
$$L\cdot b =\ds\frac1N\sum_{i=1}^N b^i M(S,s_i)\cdot\Big( f(s_{i}, \hat{y}(s_i), \w_i)- f(s_{i}, \hat{y}(s_i), \hat{w}(s_i)\Big), \quad\forall b\in \R^N.$$

Notice that, because of the continuity  w.r.t. $\epsilon$ of the left-hand side of \eqref{ab}, for every $\delta>0$, the map  $\epsilon\mapsto A^{\delta}(\epsilon)$ is continuous. 
By rewriting relation \eqref{ab}  as
$$ y\Big[\theta^{\delta^2}_{\hat{w},w_{1}^{\delta/N},\ldots, w^{\delta/N}_N}\left(\frac{\epsilon}{\delta}\right)\Big](S)=\hat{y}(S)+ L\cdot\epsilon + h^{\delta}(\epsilon),$$
 we get 
$$\hat{y}(S)+ L\cdot\epsilon + h^{\delta}(\epsilon)\in F(\epsilon),$$
which means that  $L$  is a  QDQ of $F$ at $(0,\hat y(S))$ in the direction of the set  $\Gamma_N$.
 Therefore, $\Lambda$ is  also a QDQ
   of $F$ at  $(0,\hat y(S))$ in the direction of the cone  $\Gamma=[0,+\infty[^N $. Since
  $ {\bf C}_{\mathfrak{w}_{1},...,\mathfrak{w}_{N} }^{s_{1},...,s_{N}} = L\cdot \Gamma$,
  one concludes that  $ {\bf C}_{\mathfrak{w}_{1},...,\mathfrak{w}_{N} }^{s_{1},...,s_{N}}$ is a QDQ approximating cone to   
 $\mathcal{R}_{\V}\cup\{ \hat{y}(S)\}$ at $\hat{y}(S)$.\end{proof}

\begin{lma}\label{claim}
  Fix $\gamma\in\Gamma_N$. Then, the map
$\epsilon\to y_{{\gamma}}\Big[\hat w, w^{\epsilon^1}_1,...,w_N^{\epsilon^N}\Big](S)$
verifies 
\begin{equation}\label{mezzavia}\begin{array}{l} y_{{\gamma}}\left[\hat w, w^{\epsilon^1}_1,...,w_N^{\epsilon^N}\right](S) -\hat y(S) \\ \ \ =\ds\sum_{i=1}^N\gamma^{i} \epsilon^{i} M(S,s_{i}) \cdot \Big( {f}(s_{i},\hat{y}(s_{i}), \mathfrak{w}_{i})-{f}(s_{i},\hat{y}(s_{i}), \hat{w}(s_{i})) \Big)+\phi(\epsilon,  \gamma),\end{array}
\end{equation}
where $\phi$ is a continuous function verifying  $ \max \left\{ \phi (\epsilon, \gamma):\; \gamma\in \Gamma_N \right\} = o(\epsilon)$.
\end{lma}
\begin{proof}
Let us begin proving  the lemma in the case when $N=1$ and  $0\leq\gamma \leq1$.  One has 
\bel{STIMA!}\begin{array}{l} 
 y_{\gamma}\left[\hat w, w^{\epsilon^1}_1\right](s_1) -\hat y(s_1)= y_{{\gamma}}\left[\hat w, w^{\epsilon^1}_1\right](s_1-\epsilon^1) -\hat y(s_1-\epsilon^1)
 \\
 \qquad +  \ds\int_{s_1-\epsilon^1}^{s_1}\left( f_{\gamma}\left(s,y_{{\gamma}}[\hat w, w^{\epsilon^1}_1](s), \hat w(s),\w_1\right) - f(s,\hat y(s), \hat w(s)) \right)ds\\
 \qquad = \ds\int_{s_1-\epsilon^1}^{s_1}\left( f_{\gamma}\left(s,y_{{\gamma}}[\hat w, w^{\epsilon^1}_1](s), \hat w(s),\w_1\right) - f(s,\hat y(s), \hat w(s)) \right)ds\\
= \epsilon^1\Big( f_{\gamma}\left(s_1,\hat y(s_1),\hat w(s_1), \w_1\right) - f(s_1,\hat y(s_1), \hat w(s_1)) \Big) +\Phi_1(\epsilon^1,  \gamma) + \Phi_2(\epsilon^1)\\\\
=\gamma\, \epsilon^1\cdot \Big(f\left(s_1,\hat y(s_1),\hat \w_1\right)-f\left(s_1,\hat y(s_1),\hat w(s_1)\right)\Big) +\Phi_1(\epsilon^1,  \gamma) + \Phi_2(\epsilon^1),
 \end{array}
\eeq
where 
 $$\Phi_1(\epsilon^1,  \gamma) := \ds\int_{s_1-\epsilon^1}^{s_1}\left( f_{\gamma}\left(s,y_{{\gamma}}[\hat w, w^{\epsilon^1}_1](s), \hat w(s),\w_1\right) - f_{\gamma}\left(s_1,\hat y(s_1),\hat{w}(s_1), \w_1\right) \right)ds,
 $$ 
 $$
 \Phi_2(\epsilon^1) := \ds\int_{s_1-\epsilon^1}^{s_1}\Big( f(s_1,\hat y(s_1), \hat w(s_1))  - f(s,\hat y(s), \hat w(s)) \Big)ds.
 $$
To simplify the notation, in what follows we will write $y_{\gamma}(s)$ in place of $y_{{\gamma}}[\hat w, w^{\epsilon^1}_1](s)$. Using hypothesis \textbf{(SH)}-(ii), one  obtains the following estimate:  
\begin{equation}
\begin{array}{llll}
\\
\ds \left|y_{\gamma}(s_1)-\hat{y}(s_1) \right|\leq \int^{s_1}_{s_1-\epsilon^1}\left| f(s,y_{\gamma}(s),\hat{w}(s))-f(s,\hat{y}(s),\hat{w}(s)) \right|ds+\\
\ds \gamma\int^{s_1}_{s_1-\epsilon^1}\left| f(s,y_{\gamma}(s),\w_1)-f(s,y_{\gamma}(s),\hat{w}(s))+f(s,\hat{y}(s),\hat{w}(s))-f(s,\hat{y}(s),\w_1) \right| ds+\\
\ds \gamma\int^{s_1}_{s_1-\epsilon^1}\left|f(s,\hat{y}(s),\w_1)-f(s,\hat{y}(s), \hat{w}(s)) \right|ds\leq \\
\ds \left(1+2\gamma \right)\int^{s_1}_{s_1-\epsilon^1} c(s) \left|y_{\gamma}(s)-\hat{y}(s) \right| ds +\gamma\int^{s_1}_{s_1-\epsilon^1}\left|f(s,\hat{y}(s),\w_1)-f(s,\hat{y}(s), \hat{w}(s)) \right|ds.
\end{array}
\end{equation} 
Setting, for $s\in [s_1-\epsilon^1, s_1]$, 
\begin{equation}
\alpha(s)=\gamma\,\int^{s}_{s_1-\epsilon^1}\left|f(s,\hat{y}(s),\w_1)-f(s,\hat{y}(s), \hat{w}(s))\right|ds,
\end{equation}
it follows from the Gronwall's Lemma that
\begin{equation}\label{StimaGronwall1}
\begin{array}{ll}
\ds \left|y_{\gamma}(s_1)-\hat{y}(s_1) \right|\leq \alpha(s_1)+\int_{s_1-\epsilon^1}^{s_1}(1+2\gamma)c(s) \mathrm{exp}\left\{(1+2\gamma)c(s)(s_1-s) \right\} \alpha(s) ds\leq \\
\ds \leq 2\gamma\,\int_{s_1-\epsilon^1}^{s_1} c(s) ds+ o(\epsilon^1) \rightarrow 0
\end{array}
\end{equation}
when $\epsilon^1\rightarrow 0$. Therefore,
\begin{equation}\label{StimaGronwall2}
\left| y_{\gamma}(s)-\hat{y}(s_1) \right|\leq\left| y_{\gamma*}(s)-y_{\gamma*}(s_1) \right|+\left| y_{\gamma*}(s_1)-\hat{y}(s_1) \right| \leq (1+4\gamma)\int_{s_1-\epsilon^1}^{s_1} c(s) ds+ o(\epsilon^1).
\end{equation}
Since 
\begin{itemize}
 \item $s_1$ is a Scorza-Dragoni point of ${f}(\cdot,\cdot, \hat{w}(\cdot)) $ and ${f}(\cdot,\cdot,\w_1)$ ,  and 
\item the maps $y\mapsto {f}(s, y, \hat{w}(s))\}$, $y\mapsto{f}(s,y,\w_1)$ are Lipschitz continuous in a neighbourhood of $\hat y([0,S])$,
\end{itemize}
  in view of \eqref{StimaGronwall1}, \eqref{StimaGronwall2}, one easily  gets 
\bel{STIMA!!}
 \max \left\{\Phi_1(\epsilon^1, \gamma):\; 0\leq \gamma \leq 1\right\}= o(\epsilon^1), \qquad \Phi_2(\epsilon^1)= o(\epsilon^1). 
\eeq
If we set  $ \phi(\epsilon^1,  \gamma):= \Phi_1(\epsilon^1,  \gamma) +  \Phi_2(\epsilon^1)$, it follows by  estimates \eqref{STIMA!} and \eqref{STIMA!!} that 
$$\begin{array}{l}
 y_{\gamma}\left[\hat w, w^{\epsilon^1}_1\right](s_1) -\hat y(s_1)\\
=
   \gamma\,\epsilon^1 \Big(f\left(s_1,\hat y(s_1),\hat \w_1\right)-f\left(s_1,\hat y(s_1),\hat w(s_1)\right)\Big) +\phi(\epsilon^1,  \gamma). \end{array}
$$
Hence, one has
$$
\frac{d}{d\epsilon^1}  y_{\gamma}\left[\hat w, w^{\epsilon^1}_1\right](s_1)_{\ds |_{\epsilon^1=0}} = 
 \gamma\, \Big(f\left(s_1,\hat y(s_1),\hat \w_1\right)-f\left(s_1,\hat y(s_1),\hat w(s_1)\right)\Big),   $$which, by the basic  theory of linear  ODE's, implies that 
$$\begin{array}{l}
\ds\frac{d}{d\epsilon^1}  y_{\gamma}\left[\hat w, w^{\epsilon^1}_1\right](S)_{{\ds |_{\epsilon^1=0}}} =
M(S,s_1)\cdot \frac{d}{d\epsilon^1}  y_{\gamma}\left[\hat w, w^{\epsilon^1}_1\right](s_1)_{ \ds |_{\epsilon^1=0}}=\\
M(S,s_1)\cdot \gamma\Big(f\left(s_1,\hat y(s_1),\hat \w_1\right)-f\left(s_1,\hat y(s_1),\hat w(s_1)\right)\Big).\end{array} $$
Therefore, the lemma is proved for $N=1$ and for any {$0\leq\gamma\leq1$}.
The  general case $N\geq2$ is easily obtained  by a  finite induction argument.  The latter
doesn't present any new difficulty with respect to  the proof of the case $N=1$.  Moreover is almost verbatim the one utilized in the proof of the Pontryagin Maximum Principle  when passing from single to  multiple, finitely many needle variations  (see e.g. \cite{Scha}, Theorem 4.2.1). For this reason we omit it.
 \end{proof}
%
%
 
The reason why we have adopted QDQ approximating cones as tangential objects relies on  the validity of the following result. \footnote{Such a result would be not true if we chose to utilize AGDQ approximating cones}

\begin{thm}\label{propcomp}  Let the  original family of controls $\mathcal{V}\subset \mathcal{W}$ be abundant in  $\mathcal{W}$, and 
let  a feasible extended process $(\hat y,\hat w)$ be given. Consider a  positive  integer $N$, $N$ control values  $\mathfrak{w}_{1},...,\mathfrak{w}_{N} \in \mathfrak{W}$, and $N$ instants
  $ s_{1},...,s_{N}\in \mathrm{SD}\{{f}(\cdot,\cdot, \hat{w}(\cdot))\}\cap 
\mathrm{SD}\{{f}(\cdot,\cdot,\w_1)\}\cap\ldots\cap 
\mathrm{SD}\{{f}(\cdot,\cdot,\w_N)\}$,
$0< s_1<\ldots,< s_N\leq S$ .
Moreover, let   ${C}$ be  a  QDQ approximating cone    to the target $\cruk$ at $\hat{y}(S)$.
If  ${\bf C}_{\mathfrak{w}_{1},...,\mathfrak{w}_{N} }^{s_{1},...,s_{N}} $ and $ C$,  are  
complementary subspaces, i.e.  ${\bf C}_{\mathfrak{w}_{1},...,\mathfrak{w}_{N} }^{s_{1},...,s_{N}} \oplus C=T_{\hat{y}(S)}\M$,  then there exists a sequence $(z_k)_{k\in\mathbb{N}} \subset {\mathcal R}_{\V}\cap \cruk$ such that 
 $$\lim_{k\to\infty} z_k =\hat{y}(S)$$.
\end{thm}

\begin{proof}In view of Theorem \ref{apprpicc}, ${\bf C}_{\mathfrak{w}_{1},...,\mathfrak{w}_{N} }^{s_{1},...,s_{N}} $ is a  QDQ approximating cone to $\mathcal{R}_{\V}\cup\{\hat{y}(S)\}$ at $\hat{y}(S)$. Furthermore, since $C$ is a QDQ approximating cone to the target $\cruk$ at $\hat{y}(S)$, there   exist a positive integer   $M$, a set-valued
map $G :  \R^M  \rightsquigarrow \M$, a convex cone  $\Gamma\subset\R^M$, and a Quasi Differential Quotient $L$ of $G$ at  $(0,\hat{y}(S))$ in the direction of $\Gamma$
such that $G(\Gamma)\subseteq \cruk$ and $ C =L\cdot\Gamma$. In order to conclude the proof, it is enough to apply Theorem \ref{teoteo}, $ii)$, with $C_1= {\bf C}_{\mathfrak{w}_{1},...,\mathfrak{w}_{N} }^{s_{1},...,s_{N}}$, $C_2=C$,  $N_1=N$, $N_2=M$, $\Gamma_1=[0,\infty)^N$, $\Gamma_2=\Gamma$, $F_1$ defined as in \eqref{F(eps)}, and $F_2=G$. This concludes the proof.
\end{proof}

\section{The main result}

 \begin{thm}[{\sc A geometric principle for gaps}]\label{GEOM} 
  Let us assume that the family of controls  $\V$ is abundant in $\W.$  Let $(\hat y,\hat w)$ be a feasible extended process   satisfying  the infimum gap condition.
  Then  any QDQ  approximating cone $C$    to $\cruk$  at $\hat y(S)$ is { \it  linearly separable}  from  any  extended variational cone ${\bf C}_{\mathfrak{w}_{1},...,\mathfrak{w}_{N} }^{s_{1},...,s_{N}}$, i.e.
  there exists a non-zero linear form $\xi\in T^*_{\hat y(S)}\mathcal{M}$ such that
  $$
  \xi\cdot c_1\leq 0 \leq  \xi\cdot c_2, \qquad \forall (c_1,c_2)\in {\bf C}_{\mathfrak{w}_{1},...,\mathfrak{w}_{N} }^{s_{1},...,s_{N}} \times C. 
  $$
\end{thm}

Let us give   the definition of {\it abnormal extremal, normal  $h$-extremal, and $h$-abnormal extremal.}

{\begin{definition}[Abnormal extremal]\label{abextremalityn} Let  $(\hat y, \hat w)$  be a feasible extended process,
   and  let   ${C}$ be  a  QDQ approximating cone    to the target $\cruk$ at $\hat{y}(S)$.
We say that the process $(\hat{y},\hat{w})$ is  an {\rm abnormal   extremal} (with respect to  $C$)   if  there exists a 
  lift $(\hat y,\lambda)\in  W^{1,1}([0,S];\,T^*\mathcal{M})$  of   $\hat y$  verifying the following conditions:  
\begin{itemize}
\item[$(i)\;$]$\ds \frac{d\lambda}{ds} =  -\ds\lambda\cdot\frac{\partial f}{\partial y}(s,\hat y(s),\hat w(s))$;
 \item[$(ii)\;$] 
 $\ds \max_{\w\in\FW}  {\lambda}(s)\cdot{f}(s,\hat{y}(s),\w)={\lambda}(s)\cdot{f}(s,\hat{y}(s),\hat{w}(s))$ \  \ 
 a.e. $s\in [0,S]$;
\item[$(iii)\;$]  $ {\lambda}(S)\in- {C}^{\perp}$;
 \item[$(iv)\;$] $ \lambda \neq 0$.
 \end{itemize}
\end{definition}

{\begin{definition}[$h$-extremal]\label{extremalityn} Let  $(\hat y, \hat w)$  be a feasible extended process, let  a  (cost) function $h:{\mathcal M}\to\rr$ be  differentiable at $\hat y(S)$,
   and  let   ${C}$ be  a   QDQ approximating cone    of the target $\cruk$ at $\hat{y}(S)$.
We say that the process $(\hat{y},\hat{w})$ is  an {\rm  $h$-extremal} (with respect to $h$ and  $C$)   if  there exist a 
  lift $(\hat y,\lambda)\in  W^{1,1}([0,S];\,T^*\mathcal{M})$  of   $\hat y(\cdot)$ and  a {\rm cost multiplier}  $\lambda_c\in\{0,1\}$ such that:  
\begin{itemize}
\item[$(i)\;$]$\ds \frac{d\lambda}{ds} =  -\ds\lambda\cdot\frac{\partial f}{\partial y}(s,\hat y(s),\hat w(s))$;
 \item[$(ii)\;$] 
 $\ds \max_{\w\in\FW}  {\lambda}(s)\cdot{f}(s,\hat{y}(s),\w)={\lambda}(s)\cdot{f}(s,\hat{y}(s),\hat{w}(s))$ \  \ 
 a.e. $s\in [0,S]$; 
\item[$(iii)\;$]  $ {\lambda}(S)\in -\lambda_c \nabla h(\hat{y}(S))- {C}^{\perp}$;
 \item[$(iv)\;$] $ {(\lambda,\lambda_c)}\neq 0$.
 \end{itemize}
 
Furthermore, we say that  an {\rm $h$-extremal $(\hat y, \hat w)$ is normal} if for every  choice of the pair $(\lambda,\lambda_c)$, one has  $\lambda_c=1$.  We  say that  an {\rm  $h$-extremal $(\hat y, \hat w)$ is abnormal} if it is not normal, namely if exists  a choice of $(\lambda,\lambda_c)$ with $\lambda_c=0$.
\end{definition}

\begin{remark}{\rm Though these definitions have  intrinsic meanings, we have chosen to adopt a notation reminiscent of coordinates. Of course,
the adjoint equation $(i)$ might be expressed --when coupled with the dynamics-- as the Hamiltonian system
$$
\frac{d}{dt}(y,\lambda) = X_H(s,x,\lambda):=J\cdot DH(s,x,\lambda),
$$
where $H:T^*\N\to \R$ is the maximized Hamiltonian defined  by setting 
$$
H(s,x,\lambda):= \max_{\w\in\FW} {\lambda}(s)\cdot{f}(s,x,\w) \qquad \forall (s,x,\lambda)\in [0,S]\times T^*\N
$$
and  $X_H$ is the Hamiltonian vector field, namely $X_H:=J\cdot DH$, $J$ being the symplectic matrix and $D$ the differential operator with respect to $x$ and $\lambda$.

Let us also point out that the dot $\cdot$ has obvious  different meanings according to the context: for instance,  in $(i)$ of the definitions above, it denotes a linear operator on the  cotangent space,  while in $(ii)$ it denotes the duality product. }
\end{remark}

 \if{
 
  For this reason, let us  recall the notion of extremal on a differential manifold.

{\begin{definition}\label{extremality} Let  $(\hat y, \hat w)$  be a feasible extended process, let   ${C}$ be  a Boltyanski approximating cone    of the target $\cruk$ at $\hat{y}(S)$, and let $h:{\mathcal M}\to\rr$ be a differentiable (cost) function  and assume that 
\begin{itemize}
\item[{\bf (D)}] there exists $\delta>0$ s.t., for each $(s,\w)\in [0,S]\times \mathfrak{W}$, the mapping $y\mapsto f(s,y,\w)$  is differentiable  for every $y\in \hat{y}(s)+\delta B$.
\end{itemize} 
We say that the process $(\hat{y},\hat{w})$ is  a $(h,C)$-{\rm extremal}    if  there exist an absolutely continuous    path ${\lambda}(\cdot)\in  W^{1,1}([0,S];\,T^*\mathcal{M})$ and  a cost multiplier  $\lambda_c\in\{0,1\}$  such that {\rm $ {\lambda}(\cdot)$  is a lift of   $\hat y(s)$} --i.e. $\pi\circ{\lambda}(\cdot) = \hat y(\cdot)$-- and, moreover, the following conditions hold true:  

\begin{itemize}
\item[$(i)\;$] $( {\lambda}(\cdot),\lambda_c)\neq 0\qquad$ \,\, \,\,\,\,\,\,\,\,\,\,\,\,\,\,\,\,\,\,\,\, \,\,\,\,\,\,\,\,\,\,\,\,\,\,\,\,\,\,\,\,\,\,\,\,\,\,\,\,\,\,\,\,\,\,\,\,\,\, \,\,\,\,\,\,\,\,\,\,\,\,\,\,\,\,\,\,\,\, \,\rm{(Non-triviality)};
\item[$(ii)\;$] $\ds \frac{d\lambda}{ds} =  X_H(\lambda(s))  \quad a.e.\; s\in [0,S]{\rm \qquad\qquad\qquad\qquad (Hamilton\; Equations),}$

 where $ X_H$ denotes the Hamiltonian vector field\footnotemark
 corresponding to  $H$ .\footnotetext{ As is well known, the Hamiltonian vector field $X_H$ is defined on $T^*\M$ by the relation
$ dH(Y)=\omega(X_H, Y)$, where $Y$ is any vector field on $T^*\M$ and $\omega$ is the simplectic form.}

   

 \item[$(iii)\;$] 
 $  {\lambda}(s)\cdot{f}(s,\hat{y}(s),\w) - 
  {\lambda}(s)\cdot{f}(s,\hat{y}(s),\hat{w}(s)) 
  \leq 0\quad  \,\,\,\,\,\,\,\,\,\,\,\,\,\, \rm{(Maximum\; Principle)} $

 for every $\w\in\FW$ and every $s\in [0,S]\backslash\mathcal{N}$;
\item[$(iv)\;$] $ {\lambda}(S)\in -\lambda_c \nabla h(\hat{y}(S))- {C}^{\perp},$
 \,\,\,\,\,\,\,\,\,\,\,\,\,\,\,\,\,\,\,\,\,\,\,\,\,\,\,\,\,\,\,\,\,  \,\,\,\,\,\,\,\,\,\,\,\,\,\,\,\,\,\,\,\,\,\,\,\,\,\,\,\,\,{\rm (Non-transversality)}.
 \end{itemize}\end{definition}
 }\fi
 
%
%

\if{
\begin{definition}
 We say that a  feasible process $(\hat{y},\hat{w})$ is {\rm a normal $(h,C)$-extremal } if any choice $( {\lambda}(\cdot),\lambda_c)$ in the above definition verifies $\lambda_c \neq 0$
\end{definition}

\begin{definition}\label{abndef}
 We say that a feasible process  $(\hat{y},\hat{w})$ is {\rm an abnormal $C$-extremal } if there exists an absolutely continuous  path $ {\lambda}(\cdot)\in  W^{1,1}([0,S];\,T^*\mathcal{M})$ such that ${\lambda}\neq 0$, conditions  ($i$), ($ii$) and ($iii$) in Definition \ref{extremality} hold true, and, moreover, 
\begin{itemize} \item[$(iv)_0\;$] ${\lambda}(S)\in - {C}^{\perp}$.\qquad\footnotemark
 \end{itemize}
\end{definition}
}\fi

%
%
%

Observe that every abnormal extremal is an abnormal $h$-extremal  for any cost $h$ differentiable at $\hat y(S)$, while every abnormal $h$-extremal is an abnormal extremal.  We are now ready to state our main result on infimum gaps .

\begin{cor}[{\sc Normality No-Gap Criterion}]\label{MP_INF_GAP} 
  Let us assume that the family of controls  $\V$ is abundant in $\W.$  If a feasible extended process $(\hat y,\hat w)$  satisfies the infimum gap condition,
 then, for every  QDQ approximating cone     $C$ to $\cruk$ at $\hat y(S)$,    $(\hat y,\hat w)$ is an  abnormal  extremal  with respect to $C$.
\end{cor}

%
%

When referred to a specific cost $h$, the  contrapositive version of this theorem provides a sufficient condition for the absence of local infimum  gaps. Precisely:

\begin{cor}[\sc A sufficient condition for avoiding infimum gaps]\label{mainth} 
Let us assume that the family of controls  $\V$ is abundant in $\W$, and let   $(\hat y,\hat w)$  be a feasible extended process. Let   $h:{\mathcal M}\to\rr$ be a cost function, differentiable  at $\hat y(S)$, and let $(\hat y,\hat w)$ be a normal $h$-extremal  for some  { QDQ} approximating cone     $C$ to $\cruk$ at $\hat y(S)$.
 Then there is no infimum gap at $(\hat y,\hat w)$. 
\end{cor}

As we have  mentioned in the Introduction, the  relation between gap phenomena and abnormality has been quite investigated   in two cases of embeddings: the  embedding of  bounded optimal control problems into their {\it convex relaxation} \cite{PV1, PV2, PV3} and the embedding  of unbounded (convex) control systems  into  their  {\it impulsive, space-time  closure} \cite{MRV}.
Since the original control families in such embeddings turn out to be abundant in their extensions, these kinds of results can be also obtained  by Theorem \ref{MP_INF_GAP}.\footnote{Although the use of different types of cones  describing the non-transversality condition makes  Theorem \ref{MP_INF_GAP} and the results in \cite{ MRV, PV1, PV2, PV3} distinct (see \cite{PRCDC} for the details).}
 In Section \ref{Ex_Imp}, we are going to present a new application to a dynamics which is neither convex nor bounded.

\subsection{A verifiable sufficient condition for normality}

In practical situations, it may be difficult or even impossible to directly verify the normality of an extremal, which, in view of Corollary \ref{mainth}, would guarantee the absence of gaps. This motivates Theorem \ref{normsuff} below, which provides a  sufficient condition {\it on the data of the problem} in order for the  extremals to  be normal. 

In the following definition we assume that a control system 
\bel{cs}\displaystyle \ds\frac{dy}{ds}(s)={f}(s, y(s),w(s)),\qquad w\in\W, \qquad \mathrm{a.e.}\quad s\in [0,S],
\eeq
as above is given, with an initial condition 
\bel{ic}
y(0)=\bar y,
\eeq
and, still, we use $\mathcal{R}_\W$ to denote the reachable set from $\bar y$.
 
 \begin{definition} Consider  a point $\tilde y\in\cruk$ and let   $C\subset T_{\tilde y}\N$  be a QDQ approximating  cone of $\cruk$ at $\tilde y$. If $S>0$, we say that the point  $\tilde y$ is {\rm   $C$-needle-controllable} (with respect to \eqref{cs}-\eqref{ic}) at $S$,
if, for  every ${\boldsymbol\xi}\in C^\bot \backslash \{0\}$, there exist $\delta_1>0$ and $\delta_2\in (0,S]$   such that   
\bel{nr}
\inf_{\check\w\in \FW}\sup_{\w{\in{\FW}} } {\boldsymbol\xi}\cdot ({f}(s,\tilde y,\w)  - {f}(s,\tilde y,\check\w)) \geq \delta_1\quad\quad {a.e.}\; {s\in [S-\delta_2,S]}.
\eeq
\end{definition}

\begin{thm}\label{normsuff} Consider a feasible process $(\hat y,\hat w):[0,S]\to\N\times\FW$  of \eqref{cs}-\eqref{ic}. 
 Let $C$ be a  QDQ  approximating cone to $\cruk$  at   $\hat y(S)$, and let $\hat y(S)$ be {\rm   $C$-needle-controllable}  at $S$.
Then the  process $(\hat y,\hat w)$ is not  an abnormal extremal, so, in particular,  it does not satisfy the infimum-gap condition.

\end{thm}

\begin{proof}
Assume by contradiction that the extremal $(\hat y,\hat w)$ is abnormal, namely that
 there exists an absolutely continuous lift   $({\hat y,\lambda}):[0,S]\to T^*\M$  of $\hat y$ such that $\lambda\neq 0$, 
$
{\lambda}(S)\in -C^{\bot}
$, and  the inequality  
\bel{mpvero}
{\lambda}(s) \cdot  ({f}(s,\hat y(s),\w)  - {f}(s,\hat y(s),\hat w(s))) \leq 0
\eeq
holds true for almost  every $s\in [0,S]\backslash I_0$ and every $\w\in\FW$, 
$I_0$ having Lebesgue measure equal to zero.
Taking ${\boldsymbol\xi}:={\lambda}(S)$  in  \eqref{nr},  we deduce
 that there exist $\delta_1,\delta_2>0$  and a neighbourhood $U\subset T^*\N$ of $( \hat y(S),{\boldsymbol\xi})$ such that, for all $(y,p)\in U$, 
$$
\sup_{\w{\in{\FW}} } { p}\cdot ({f}(s,y,\w)  - {f}(s,y,\check\w)) > \frac{\delta_1}{2}\quad \forall \check\w\in \FW,\quad  {a.e.}\; {s\in [S-\delta_2,S]}.
$$
Now, by choosing $\varepsilon\in]0,\delta_2]$ sufficiently small, for every $s\in [S-\varepsilon,S]$ one   has $(\hat y(s),{\lambda}(s)) \in U$, so that 
$$
\sup_{\w{\in{\FW}} }{\lambda}(s)\cdot ({f}(s,\hat y(s),\w)  - {f}(s,\hat y(s),\check\w)) > \frac{\delta_1}{2} \quad \forall \check\w\in \FW, \quad \hbox{a.e.}\; {s\in [S-\varepsilon,S]}.
$$
In particular, for all $s\in[S-\varepsilon,S]\backslash  I_0$,
$$
\sup_{\w{\in{\FW}} } {\lambda}(s)\cdot ({f}(s,\hat y(s),\w)  - {f}(s,\hat y(s),\hat w(s))) > 0,
$$
which contradicts the maximization relation \eqref{mpvero}.

\end{proof}

\section{Proofs of the main results }
\subsection{Proof of the Geometric Principle  (Theorem   \ref{GEOM})}
%
%
By a basic result on control system  (see e.g. \cite{Scha}, \cite{BP}), ${\bf C}_{\mathfrak{w}_{1},...,\mathfrak{w}_{N} }^{s_{1},...,s_{N}}$ turns out to be 
  a QDQ approximating cone   to  the (local) extended reachable set $\mathcal{R}^{\hat w,r}_\W$ at $\hat y(S)$.\footnote{For instance: it is well-known that ${\bf C}_{\mathfrak{w}_{1},...,\mathfrak{w}_{N} }^{s_{1},...,s_{N}}$ is a Boltyanski approximating cone to  $\mathcal{R}^{\hat w,r}_\W$ at $\hat y(S)$  (see e.g. \cite{Suss1}). Furthermore, a Boltyanski approximating cone is clearly a QDQ approximating cone.}
More importantly,  Theorem \ref{apprpicc}  states  that  ${\bf C}_{\mathfrak{w}_{1},...,\mathfrak{w}_{N} }^{s_{1},...,s_{N}}$ is also a QDQ approximating cone to  the (local) original  reachable set $\mathcal{R}^{\hat w,r}_\V  \cup \{ \hat{y}(S)\} $ at $\hat y(S)$. 
Therefore, by Lemma \ref{lemma infgap}, the sets  $\Big(\mathcal{R}_\V^{\hat w,r}\cup \{ \hat{y}(S)\}\Big)$ and $\cruk$ are locally separated at $\hat{y}(S)$, which by Theorem \ref{teoteo}, $ i)$, implies that the cones   ${\bf C}_{\mathfrak{w}_{1},...,\mathfrak{w}_{N} }^{s_{1},...,s_{N}}$ and $C$ are not strongly transverse. Since linear separability is equivalent to non-transversality (Proposition \ref{teo3})  we have to prove that {\it these cones  are   not transverse as well}. Indeed, in view of Proposition 
\ref{teo2} the only case in which  they might happen to be transverse (and not strongly transverse) is the one in which the cones are complementary subspaces of $T_{\hat y(S)}\N$. 
 However, such an instance is excluded by Theorem \ref{propcomp} and the occurrence of an infimum gap. In fact, if ${\bf C}_{\mathfrak{w}_{1},...,\mathfrak{w}_{N} }^{s_{1},...,s_{N}}$ and $C$ satisfy ${\bf C}_{\mathfrak{w}_{1},...,\mathfrak{w}_{N} }^{s_{1},...,s_{N}}\oplus C=T_{\hat y(S)}\N$, then Theorem \ref{propcomp} assures the existence of a sequence $(y_k)_{k\in\mathbb{N}}\subset\mathcal{R}_\V\cap \cruk $ such that
$y_k\to 
\hat y(S)$, which contradicts  the fact that 
$(\hat y,\hat w)$ verifies the infimum gap condition. This concludes the proof.\qquad \qquad \qquad\square


 \subsection{Proof of the Normality No-Gap Criterion (Corollary
  \ref{MP_INF_GAP})}

  By Theorem \ref{GEOM}, the cones ${\bf C}_{\mathfrak{w}_{1},...,\mathfrak{w}_{N} }^{s_{1},...,s_{N}}$ and $C$  are linearly separable. This means that 
there exists  $\xi\in(T_{\hat{y}(S)}\mathcal{M})^*\backslash\{0\}$ such that  $\xi\in -C^{\perp}\cap \left( {\bf C}_{\mathfrak{w}_{1},...,\mathfrak{w}_{N} }^{s_{1},...,s_{N}}\right)^{\perp}$.
  Now let us set  $\lambda(s):= \xi\cdot M(S,s)$, where $M(S,s)$ is the fundamental  matrix defined in \eqref{linearized_problem},
 so that  
$$\lambda\neq 0, \qquad \lambda(S)\in -C^{\perp},\quad \ds \frac{d\lambda}{ds}(s)= -\lambda(s)\cdot\frac{\partial{f}}{\partial y}(s,\hat{y}(s),\hat{w}(s)), \quad\hbox{\rm for  a.e.}\,\,\, s\in[0,S].$$
By $\xi\in \left( {\bf C}_{\mathfrak{w}_{1},...,\mathfrak{w}_{N} }^{s_{1},...,s_{N}}\right)^{\perp}$, it follows that, for every $i=1,\ldots,N$, 
\bel{quasi}\begin{array}{l}
0\geq \xi \  \cdot  \left( M(S,s_i)\cdot \Big({f}(s_i,\hat{y}(s_i),\w_i) - 
 {f}(s_i,\hat{y}(s_i),\hat{w}(s_i)) \Big)\right) \\
 =\Big(\xi \cdot M(S,s_i)\Big)\  \cdot \ \Big( {f}(s_i,\hat{y}(s_i),\w_i) - 
 {f}(s_i,\hat{y}(s_i),\hat{w}(s_i)) \Big) \\ = \lambda(s_i)\cdot\Big({f}(s_i,\hat{y}(s_i),\w_i) - 
 {f}(s_i,\hat{y}(s_i),\hat{w}(s_i)) \Big).
  \end{array}\eeq
 Therefore the lift $(\hat y,\lambda)$ verifies ($i$)-($iv$)  of Definition \ref{abextremalityn}, except that  ($iii$) is verified only for every finite set of pairs $(s_i,\w_i)\in [0,S]\times \FW$, $i=1,\ldots, N$, such that 
  $ s_{1},...,s_{N}\in \mathrm{SD}\{{f}(\cdot,\cdot, \hat{w}(\cdot))\}\cap 
\mathrm{SD}\{{f}(\cdot,\cdot,\w_1)\}\cap\ldots\cap 
\mathrm{SD}\{{f}(\cdot,\cdot,\w_N)\}$,  $0< s_1<\ldots,< s_N\leq S$.
To conclude the proof  we have to show the validity of ($iii$) in the whole control value set $\FW$ and almost all times. This is achieved through non-empty intersection arguments borrowed from   those utilized in \cite{Suss1} to prove the Maximum Principle.

\subsubsection{The case of a finite subset of controls}
   Let us consider a {\it finite subset} of control values $\bar\FW \subseteq \FW$  and  let us set
 $$ E(\bar\FW) :=\ds
\bigcap_{w\in \bar\FW}\mathrm{SD}\{{f}(\cdot,\hat{y}(\cdot),\w)\}\bigcap \mathrm{SD}\{{f}(\cdot,\hat{y}(\cdot),\hat w(\cdot))\}\quad\Big(\subset [0,S]\Big).
$$ 
Since $\bar \FW$ is finite, $ E(\bar\FW)$ has measure equal to  $S$. 
 Therefore, by Lusin's theorem we can write
\bel{lusin}\bega{l}\ds
 \ds E(\bar\FW) = \bigcup_{j=0}^{\infty}E_j,
\enda\eeq
 where $E_0$ has zero measure and, for every $j$, the set $E_j$ is
 compact, and, for every $w\in \bar\FW$, the restrictions to $E_j$ of the    map
$$
s\mapsto r^{\w}(s):={f}(s, \hat y(s),\w)-{f}(s, \hat{y}(s),\hat{w}(s)), 
$$
is continuous.
 
  For every $j$ let $D_j$ be the set of  {\it density points}
 \footnote{We recall that an element $t\in B\subset \R$ is a {\it density point for $B$} if $$\lim_{\delta\to 0+} \frac{meas([t-\delta,t+\delta])}{2\delta} = .1$$} 
of
$E_j$. Since, for every natural number  $j$, $E_j$ and $D_j$ have the same measure,
  one
obtains that \footnote{For every measurable subset $A\subseteq [0,S]$, $meas(A)$ denotes the Lebesge measurable of $A$.}
$$
meas(E(\bar\FW)) = meas(D)
$$ where we have set  $D :=\displaystyle\cup_{j=0}^{\infty}D_j$ .

Now let $F$ be an arbitrary, non-empty,  subset of $D\times \bar\FW$, and let us define the  subset 
$\Lambda(F,\bar\FW)\subseteq(T_{\hat y (S)}\mathcal{M})^*$  by setting 
$$\Lambda(F,\bar\FW):=\left\{\bar\lambda\in (T_{\hat y (S)}\mathcal{M})^*, \quad |\bar\lambda|=1,\quad \bar\lambda\;\hbox{verifies  \;{\bf (P)}$_{F}$} \right\}, $$
where property {\bf (P)$_{F}$} is as follows:
\begin{small}
\begin{itemize}

\item[\,]{\bf Property (P)$_{F}$.}   The pair $(\hat y,\lambda)\in  W^{1,1}([0,S];\,T^*\mathcal{M})$  is a lift of   $\hat y(\cdot)$  such that:
\begin{itemize}
\item[(1)] 
$\lambda(S)=\bar\lambda \in -{C}^{\bot}  \,\,$
\item[(2)] $\ds \frac{d\lambda}{ds} =  -\ds\lambda\cdot\frac{\partial f}{\partial y}(s,\hat y(s),\hat w(s)),\qquad \mathrm{a.e}\;\;s\in[0,S];$

\item[(3)] $\lambda(s)\cdot {f}(s, \hat{y}(s),\hat{w}(s)) \geq  \lambda(s)\cdot {f}(s, \hat{y}(s),\w)$   

for every $(s,\w)\in F$.

\end{itemize}
\end{itemize}
\end{small}
  \vskip0.3truecm
   
   Notice that, for every subset $F\in D\times \bar\FW$,  $\Lambda(F,\bar
\FW)$  is  compact and, moreover,
$$
\Lambda(F_1\cup F_2,\bar\FW) = \Lambda(F_1,\bar\FW)\cap \Lambda(F_2,\bar\FW)
$$ for
all $F_1,F_2\in D\times \bar\FW$.
\vskip0.4truecm

By Theorem \ref{MP_INF_GAP}, $\Lambda(F,\bar \FW) \neq \emptyset$ as soon as  $F$ is {\it finite} and  can be written as 
$$
F=\{(s_1,\w_1),..., (s_m,\w_m)\} \qquad 0\leq s_1<...<s_i<...<
s_m<S,\; \w_i\in \bar\FW.
$$

{\bf Claim:} {\it  One has  $\Lambda(F,\bar\FW) \neq \emptyset$  even when $F$ is an {\em arbitrary 
finite} subset of $D\times \bar\FW$, namely $F$ can be written  as
$$
F=\{(s_1,\w_1),..., (s_m,\w_m)\} \qquad 0\leq s_1\leq ...\leq s_i \leq ... \leq
s_m, \; \w_i\in\FW.
$$
 }

Indeed, every $s_i$ belongs to a suitable $D_h$, which can be labelled as $D_{h(i)}$. Since
 $D_{h(i)}$ is made of density points, there exist sequences $(s_{i,j})$
such that $$s_{i,j}\in D_{h(i)}\quad \forall j,\qquad   s_i = \lim_{j\to
\infty} s_{i,j}, $$ and $$s_{1,j}<...<s_{m,j}\quad \forall j\in \mathbb{N}.$$ 
Set $F_j=
\{(s_{i,j},\w_1),..., (s_{m,j},\w_m)\}$ --so that $\Lambda(F_j,\bar\FW) \neq
\emptyset$-- and choose $\bar {\lambda}_j\in
\Lambda(F_j,\bar\FW)$. Since $|\bar{\lambda}_j|=1$ for all $j$,  by possibly taking   a subsequence we can assume that   $(\bar {\lambda}_j)_{j\in \mathbb{N}}$
converges to some $\bar{\lambda}$. For every $s\in [s_1,S]$, define the lifts  $(\hat{y},\bar{\lambda}), (\hat{y}, \bar{\lambda}_j)\in W^{1,1}([s_1, S]; T\mathcal{M})$ of $\hat{y}$  such that $\bar{\lambda}(S)=\bar{\lambda}$, $\bar{\lambda}_j(S)=\bar{\lambda}_j$ and both satisfying the equation 
$$\ds \frac{d\lambda}{ds} =  -\ds\lambda\cdot\frac{\partial f}{\partial y}(s,\hat y(s),\hat w(s)),\qquad \mathrm{a.e}\;\;s\in[s_1,S].$$
The mapping $s\mapsto \bar\lambda_j(s)$  satisfies the inequality
$$\bar\lambda_j(s_{i,j})\cdot {f}(s_{i,j}, \hat{y}(s_{i,j}),\hat{w}(s_{i,j})) \geq  \bar\lambda_j(s_{i,j})\cdot {f}(s_{i,j}, \hat{y}(s_{i,j}),\w)$$
 for all $j\in \mathbb{N}$, every $i=1,\dots, m$ and $\w \in \bar\FW$. 
Since,  for every $i=1,\dots m$, the map $s\mapsto r^{\w_{i}}(s):={f}(s, \hat y(s),\w_i)-{f}(s, \hat{y}(s),\hat{w}(s))$  is continuous on  $D_{h(i)}$, the function 
 $s\mapsto \bar{\lambda}_j(s)\cdot r^{\w_{i}}(s)$
  is also continuous  on  $D_{h(i)}$, so passing to the limit  we can conclude that
$$ \bar{\lambda}(s_i)\cdot {f}(s_i, \hat{y}(s_i),\hat{w}(s_i)) \geq  \bar{\lambda}(s_i)\cdot {f}(s_i, \hat{y}(s_i),\w) $$
for every $i=1,...,m$ and $\w \in \bar\FW$.
Since one also has  $0\neq \bar{\lambda}=\bar\lambda(S) \in - {C}^{\bot}$,
the  claim is proved.

\subsubsection{The general case of an infinite control set}  
Up to  now we have shown that, {\it if $\bar\FW$ is finite}, and $F\subset D\times\bar\FW $ is {\it finite} ---and we write $card (F) <\infty$---, 
then $\Lambda(F,\bar\FW)$ is a nonempty compact set.  { We now conclude the proof throgh a standard non-empty intersection argument (see e.g. \cite{Suss1}).
If we take a finite family $F^1,\dots, F^r \subset D\times\bar\FW $ such that $card (F_i) <\infty$ for every $i=1,\dots,r$, one has
$$
\Lambda(F^1,\bar\FW)\cap\dots\cap \Lambda(F^r,\bar\FW) = \Lambda(F^1\cup\dots\cup F^r, \bar\FW) \neq \emptyset,
$$
(for $card (F^1\cup\dots\cup F^r)<\infty$).
Hence, $$\{ \Lambda(F,\bar\FW)\quad |\quad F\subset D\times\bar\FW, \quad
\hbox{card} F < \infty\}
$$
is {\it a family of compact subsets such that each finite intersection is
nonempty}.  This implies that the (infinite)  intersection of {\it all} $ \Lambda(F,\bar\FW)$  such that $card F<\infty$  is
nonempty. Therefore 
$$
\Lambda(D\times\bar\FW, \FW) = \ds \Lambda\left(\bigcup_{card(F)<\infty}F, \FW\right) = \ds \bigcap_{card(F)
 < \infty}\Lambda(F,\bar\FW)
\neq \emptyset
$$

\vskip 1truecm 

 To end the proof in the general case when  $card (\FW)$ is infinite,  for any arbitrary subset $\hat\FW\subseteq \FW$  define
$$\Lambda(\hat \FW):=\left\{\bar\lambda\in (T_{\hat y (S)}\mathcal{M})^*, \quad |\bar\lambda|=1,\quad \bar\lambda\;\hbox{verifies  \;{\bf (PP)}$_{F}$} \right\},$$
 where property  {\bf (PP)}$_{\hat{F}}$ is as follows:

\begin{itemize}
\item[\,]{\bf (PP)$_{\hat{F}}$:} 
The pair $(\hat y,\lambda)\in  W^{1,1}([0,S];\,T^*\mathcal{M})$  is a lift of   $\hat y(\cdot)$  such that:
\begin{itemize}
\item[(1)] 
$\lambda(S)=\bar\lambda \in -{C}^{\bot}  \,\,$
\item[(2)] $\ds \frac{d\lambda}{ds} =  -\ds\lambda\cdot\frac{\partial f}{\partial y}(s,\hat y(s),\hat w(s)),\qquad \mathrm{a.e}\;s\in[0,S];$
\item[(3)] For each $\w\in \hat\FW$, there exists a subset of full measure $I_{\w}\subseteq [0,S]$ such that 
$$\lambda(s)\cdot {f}(s, \hat{y}(s),\hat{w}(s)) \geq  \lambda(s)\cdot {f}(s, \hat{y}(s),\w)$$   
for every $s\in I_{\w}$, $\w\in {\hat\FW}$.
\end{itemize} 
\end{itemize}

So, proving Theorem \ref{MP_INF_GAP} is equivalent to showing that 
\bel{W}
\Lambda( \FW)\neq \emptyset.
\eeq

Since \bel{intersection}
\Lambda( \FW) = \ds \bigcap_{card(\hat \FW) < \infty} \Lambda(\hat \FW) ,
\eeq once again we have to show that the (possibly infinite) family
$$\left\{\Lambda(\hat \FW),\quad card(\hat \FW) < \infty\right\}
$$ has non-empty intersection. 
This can  easily achieved by the same arguments as above. Indeed,
 $\Lambda(\hat \FW)$ is not empty and compact as soon as $\hat\FW$ is finite. 
Furthermore, for every $\FW_1,\FW_2\subseteq \FW$ one has
$$
\Lambda(\FW_1\cup \FW_2) = \Lambda(\FW_1)\cap \Lambda(\FW_2).
$$

In particular, the family $\left\{\Lambda(\hat \FW):\; card(\hat \FW) < \infty \right\}$ is made of compact subsets and satisfies the finite intersection property, that is, the intersection of any finite  finite subfamily $\left\{\Lambda(\hat \FW):\; card(\hat \FW) < \infty \right\}$ is not empty. Therefore,  it has non-empty intersection, namely
$$
\Lambda( \FW) = \ds \bigcap_{card(\hat \FW) < \infty} \Lambda(\hat \FW) \neq
\emptyset\,\,.
$$
This concludes the proof of Theorem \ref{MP_INF_GAP}.

 \section{An application to non-convex, unbounded,  problems}\label{Ex_Imp}

Impulsive optimal control problems  --where the dynamics  is {\it  unbounded}-- have been extensively studied together with  their applications  \cite{AR, AKP1, AKP3, BR, AB1, AB2, Dy1, Dy2, GS, KDPS, annachiara, MiRu, MR, SV, WZ}.	
The space-time representation (see \eqref{spnc} below)  can be regarded as an \textit{extension} of unbounded control systems. An important case  is the one of a minimum problem with  of control-affine dynamics:
\[
(P)\left\{ \begin{array}{l}
\,\,\,\,\,\,\,\,\,\,\,\,\,\,\,\,\,\mathrm{Minimize}
\,h(t_{2}, x(t_{2}),\eta(t_{2}))\vspace{0.5cm}\\

\mathrm{over\:}\; t_{2}\in \R,\; t_{2}>t_{1},\;(x, \eta,u) \in  AC([t_1,t_2],\N\times\R)\times L^1([t_1,t_2],{U})
\\\quad\mathrm{such\: that }
\vspace{0.3cm}\\

\left\{\begin{array}{l}\ds\frac{dx}{dt}(t)=f(t,x(t))+\sum_{j=1}^{m}g_{j}(t,x(t)) u^{j}(t)\quad\mathrm{a.e.\,\, t\in [t_1,t_2]}
\\ \ds\frac{d\eta}{dt}(t)=| u(t)| \quad\mathrm{a.e.\,\, t\in [t_1,t_2]}
\end{array}\right.
\vspace{0.3cm}\\

\ds (x(t_{1}),\eta(t_{1}))=(\bar{x}, 0),\quad (t_2, x(t_2),\eta(t_{2})) \in \bar{\cruk}\times [0,K]

\end{array} , \\
\right.
\]

Here  the  set  $U$ where the controls $u$ take values is {\it  unbounded}. Furthermore,  the state $x$ range over a  $n$-dimensional Riemannian manifold $\mathcal{M}$ of class $C^{2}$, and the time-dependent  vector fields $f,g_1,\ldots,g_m$  are  of class $C^1$ in $x$, measurable in $t$,  and uniformly bounded by a $L^1$ map. 
Moreover, the cost $h:\R\times \N\times\R\rightarrow\R$ is a  continuous  function, $(t_{1},\bar{x})\in \R\times\mathcal{M}$ is a fixed initial condition,  $K$ is a non negative fixed constant, possibly equal to $+\infty$, and  the end-point constraint $\bar{\cruk}\subseteq \R\times \R^n$ is a closed subset. 
 Notice incidentally that  the function $\eta(t)$  coincides with  the $L^1$-norm of the control function $u:=(u^{1},u^{2},...,u^{m})$ on the interval $[t_1,t]$.

The  gap-abnormality  criterion  for this kind of systems  (where one considers the space-time extension \eqref{spnc} below) has been already investigated  in the case when the set of controls ${U}$ { is a convex cone \cite{MRV}}. Actually, thanks to  Corollary \ref{MP_INF_GAP}  (see also  \cite{PRCDC}),
the main result in \cite{MRV}  can be extended to  the case in which the state ranges on  a Riemannian manifold.
 However the generalization made possible by Corollary \ref{MP_INF_GAP} allows one to go much further.
Indeed, in  what follows we are  able to deduce from Corollary \ref{MP_INF_GAP}
that  the gap-abnormality criterium holds true also in  the situation when the  control set  ${U}$ is unbounded but is  {\it    neither convex  nor a cone}. 

 More precisely, we will consider  the following two cases:
\begin{itemize} 
\item[{\sc Case (i)}] 
{\it{\bf  (Space-time convex extension)}  The  controls take values on a (necessarily unbounded) subset   ${U}\subseteq \R^{m}$ such
  that {\it ${\bf co}\,{U} $ is a (convex) cone of 
$ \R^{m}$}},   where we have used  ${\bf co}\,{E}$ to denote the convex hull of a subset $E\subseteq \R^{m}$;

For instance, one could consider the set $U=\mathbb{N}^m$, so that ${\bf co}(U) = [0,+\infty[^m$.

\item[{\sc Case (ii)}] {\it {\bf  (Space-time non-convex extension)} The  controls take values on a (necessarily unbounded) subset   ${U}\subseteq \R^{m}$ such
  that 
\bel{conic}
  (r,u)\in [0,+\infty[\times U \implies \exists \rho >r \ \hbox{s.t.} \ \rho u\in U
\eeq
Notice that, if for a given set $E$ we consider the ${\bf conic}(E): =\big\{ r e \  | \ (r,e)\in [0,+\infty[\times E\big\}$  ---a cone which we call the {\it conic envelope} of $E$---,  hypothesis 
  \eqref{conic}  implies that 
  $$
  \inf_{u\in U} d(u,{\bf conic}(U)) = 0.
  $$
For instance, one could consider the set $U=\big\{(n^2,0),  (0,-m^3) \ | \ m,n\in\mathbb{N}\big\}$, 
so that  ${\bf conic}\,(U) =  [0,+\infty[\times\{0\} \bigcup\{0\} \times ]-\infty,0].$}
\end{itemize}

\begin{remark} {\rm We will treat {\sc Case (i)} in detail, describing the extension to the convex space-time system obtained by both convexification of the dynamics  and the closure of suitably reparameterized processes. Instead, we will only 
suggest the needed changes to deal with  {\sc Case (ii)}, where the only extension  comes from  reparameterization.
 However, {\sc Case (ii)} is somehow more significative, in that it marks the most important improvement with respect to the former literature initiated by Warga's work.  Indeed, in this case not only the original dynamics  but  {\it also the extended  dynamics is non-convex}. This can be of interest in those application where  the convexification of the 
 dynamics is not needed (for instance because one gets existence of minima without invoking convexification).} 
 \end{remark}

 \subsection{{\sc Case (i)} (Space-time convex extension)} 
 In order to formulate  this problem by means of  the terminology adopted in Theorem \ref{MP_INF_GAP}, we need to embed our system into  a suitably extended one. To this aim we need  to perform both a `compactification' (to manage unboundedness) and a `convexification'.  
 Let us begin by setting 
$$\begin{array}{ll}A:= \left\{a=(a^1,\ldots a^{n+3})\in{[0,1]}^{n+3},\,\, \sum_{i=1}^{n+3}a^i = 1\right\} \qquad &\mathcal{A}:=L^1\left([0,S], A \right) \\
\mathbf{W}:=\{ (w^{0},w)\in [0,\infty)\times{\bf co}\,{U}:\quad w^{0}+|w|=1\}\qquad  &\hat{ \mathcal{W}}:=L^1\left([0,S], \mathbf{W}\right) \\
\mathbf{V}:=\{ (v^{0},v)\in (0,\infty)\times {U}:\quad v^{0}+|v|=1\} \quad &\hat{ \mathcal{V}}:=L^1\left([0,S], \mathbf{V}\right)\\D:= [-0.5, 0.5]
&\mathcal{D}:=L^1\big([0,S], D\big)
\end{array}
$$

$${\W}:=\mathcal{A}\times (\hat{\W})^{n+3}\times\mathcal{D},\qquad{\V}:=\{(1,0,\ldots,0)\}\times (\hat{\V})^{n+3}\times\mathcal{D},$$
and let us consider the optimal control problem
\bel{spnc}
\,(P)_\W^h\left\{ \begin{array}{l}
\,\,\,\,\,\,\,\,\,\,\,\,\,\,\,\,\,\mathrm{Minimize}
\,h\big(z^{0}(\hat S), z(\hat S),\nu(\hat S)\big)\vspace{0.3cm}\\

\mathrm{over\:}\;(z^{0}, z, \nu, a, (w_{1}^{0},w_{1}),\ldots, (w_{n+3}^{0},w_{n+3}), d)(\cdot) \in  AC([0,\hat S],\R\times\N\times\R)\times {\W}  \times \mathcal{D}
\quad\vspace{0.3cm}\\\mathrm{s.\: t., \,\,\,for\,\, a.e. \,\,\,\,s\in[0,\hat S], }\\\,\\
\left\{\begin{array}{l}
\ds\frac{dz^{0}}{ds}(s)= (1+d(s))\sum_{i=1}^{n+3}a^{i}(s)w^{0}_{i}(s) \qquad\\
\ds\frac{dz}{ds}(s)=(1+d(s))\sum_{i=1}^{n+3}a^{i}(s)\Big(f(z^{0}(s),z(s))w^{0}_{i}(s)+\sum_{j=1}^{m}g_{j}(z^{0}(s),z(s))w^{j}_{i}(s)\Big)\quad\qquad\\
\ds\frac{d\nu}{ds}(s)=(1+d(s))\sum_{i=1}^{n+3}a^{i}(s)|w_{i}(s)| \qquad\end{array}\right.\vspace{0.3cm}\\

(z^{0}(0), z(0),\nu(0))=(0, \bar{x}, 0),\qquad (z^{0}(\hat S), z(\hat S), \nu(\hat S)) \in \bar{\cruk}\times [0,K]

\end{array}\right. . \\
\eeq
 Accordingly,  a pair
$$
\Big(\big((z^{0}, z, \nu\big)\ , \ \big(a, (w_{1}^{0},w_{1}),\ldots, (w_{n+3}^{0},w_{n+3}), d\big)\Big)
$$
such that $\big(z^{0}, z, \nu\big)$ is the solution of the above control system corresponding to the {\it control} $\big(a, (w_{1}^{0},w_{1}),\ldots, (w_{n+3}^{0},w_{n+3}), d\big)$ is called {\it a process of $(P)_\W^h$.}
The embedding of the problem $(P)$ into $(P)_\W^h$ 
 is  as follows: fix $\hat S>0$, and, for every control  $u:[t_1,t_2^u]\to U$,  consider the function $\sigma_u:[t_1,t_2^u]\to [0,\hat S]
$ defined by  
\bel{rep}\sigma_u(t):=\frac{\hat S}{t_2^u + \|u\|_1}\int_{t_{1}}^t\Big(1+  \left| u(\tau) \right| \Big)\,d\tau =\frac{\hat S}{t_2^u + \|u\|_1} (t+\eta(t)).\eeq
Then  define
 $\mathcal{I}:\R\times AC([t_1,t_2],\N\times \R \times\R^{m})\rightarrow \R \times AC([0,S],\R\times\N\times\R\times \R \times\R^{m})$ by setting
$$\mathcal{I}(x,\eta,u):=\Big(\big((z^{0}, z, \nu\big)\ , \ \big(a, (w_{1}^{0},w_{1}),\ldots, (w_{n+3}^{0},w_{n+3}), d\big)\Big)$$
where, forall  $s\in[0,\hat S]$ and all $i=1,\ldots,m$, 
$$ \qquad (z^{0},z, \nu)(s):=(id,x,\eta)\circ \sigma^{-1}_{u}(s),\qquad \forall\, s\in[0,\hat S].$$
$$ \qquad a:=(1,0,\ldots,0),\quad d:= \frac{t_2^u + \|u\|_1}{\hat S}-1,\quad (w_{i}^{0},w_{i}):=\left(\frac{1}{1+|u|}\left( 1,u \right)\right)\,\circ\,\sigma^{-1}_{u}(s).$$
By a trivial use of the chain rule one  gets the  following result (see e.g. \cite{AR} for a similar embedding):
\begin{lma}\label{characterization}  The embedding $\cal I$ is injective
\footnote{Notice that the injectivity is a consequence of the fact that  $w^0_i(s) +  \left|w_i(s)\right|= 1$ for a.e. $s\in [0,S]$ and for $i=1,\ldots,m$.}.
Moreover, the  image space of the embedding $\cal I$ coincides with the set of all processes 
$\Big(\big((z^{0}, z, \nu\big)\ , \ \big(a, (w_{1}^{0},w_{1}),\ldots, (w_{n+3}^{0},w_{n+3}), d\big)\Big)$ such that $$\big(a, (w_{1}^{0},w_{1}),\ldots, (w_{n+3}^{0},w_{n+3}), d\big)\in \V$$
%
\end{lma}

Thanks to  Lemma \ref{characterization}
we can identify the original problem $(P)$ with the problem 

\[
\,(P)_\V^h\left\{ \begin{array}{l}
\,\,\,\,\,\,\,\,\,\,\,\,\,\,\,\,\,\mathrm{Minimize}
\,h\big(z^{0}(\hat S), z(\hat S),\nu(\hat S)\big)\vspace{0.3cm}\\

\mathrm{over\:}\;(z^{0}, z, \nu, a, (w_{1}^{0},w_{1}),\ldots, (w_{n+3}^{0},w_{n+3}), d)(\cdot) \in  AC([0,\hat S],\R\times\N\times\R)\times {\V}  \times \mathcal{D}
\quad\vspace{0.3cm}\\\mathrm{s.\: t., \,\,\,for\,\, a.e. \,\,\,\,s\in[0,\hat S], }\\\,\\
\left\{\begin{array}{l}
\ds\frac{dz^{0}}{ds}(s)= (1+d(s))\sum_{i=1}^{n+3}a^{i}(s)w^{0}_{i}(s) \qquad\vspace{0.1cm}\\
\ds\frac{dz}{ds}(s)=(1+d(s))\sum_{i=1}^{n+3}a^{i}(s)\Big(f(z^{0}(s),z(s))w^{0}_{i}(s)+\sum_{j=1}^{m}g_{j}(z^{0}(s),z(s))w^{j}_{i}(s)\Big)\quad\qquad\vspace{0.3cm}\\
\ds\frac{d\nu}{ds}(s)=(1+d(s))\sum_{i=1}^{n+3}a^{i}(s)|w_{i}(s)| \qquad\end{array}\right.\vspace{0.3cm}\\

(z^{0}(0), z(0),\nu(0))=(0, \bar{x}, 0),\qquad (z^{0}(\hat S), z(\hat S), \nu(\hat S)) \in \bar{\cruk}\times [0,K]

\end{array}\right. . \\
\]
%
We can now  apply the theory  developed in the previous sections.
In view of the sufficient condition provided by Theorem \ref{Kaskosz}, it is trivial to verify that {\it the family of controls  $\V$ is  abundant in  $\W$ w.r.t. the dynamics of problem $(P)_{\W}^h$.}
Therefore, by
Corollary \ref{MP_INF_GAP},
one obtains the following  infimum-gap result:

\begin{thm}\label{gap_nc}
Consider a feasible extended process $$\Big(\big(y^{0},y,\nu\big)\ , \ \big(\hat a, (\hat w_{1}^{0},\hat w_{1}),\ldots, (\hat w_{n+3}^{0},\hat w_{n+3}), \hat d \big)\Big), \quad \hat d \equiv 0,$$ and assume that it satisfies the infimum gap condition.  Then, for all  approximating cones     $C$  to $\cruk :=\bar{\cruk}\times [0,K]$ at $(\hat{y}^{0}, \hat{y}, \hat{\nu})({\hat S})$,    there exist a number $\beta\leq0$,
 an absolutely continuous    path $(\lambda^{0}, \lambda, \lambda^{\nu})\in  W^{1,1}([0,{\hat S}];\,\rr^{(1+n+1)})$ and a zero-measure subset $I_0$  such that the following conditions hold true:

\begin{itemize}
\item[$(i)\;$] $(\lambda^{0}, \lambda, \lambda^{\nu})\neq 0\qquad$
\item[$(ii.1)\;$] $ \ds \frac{d\lambda^{0}}{ds}(s)= -\lambda(s)\cdot\left[\sum_{i=1}^{n+3}\hat{a}^{i}(s)\left(\frac{\partial f}{\partial y^{0}}(\hat{y}^{0}(s),\hat{y}(s))\hat{w}^{0}_{i}(s)+\sum_{j=1}^{m}\frac{\partial g_{j}}{\partial y^{0}}(\hat{y}^{0}(s),\hat{y}(s))\hat{w}^{j}_{i}(s)\right)\right]\,\,\,\,$
 \item[$(ii.2)\;$] $ \ds \frac{d\lambda}{ds}(s)= -\lambda(s)\cdot\left[\sum_{i=1}^{n+3}\hat{a}^{i}(s)\left(\frac{\partial f}{\partial y}(\hat{y}^{0}(s),\hat{y}(s))\hat{w}^{0}_{i}(s)+\sum_{j=1}^{m}\frac{\partial g_{j}}{\partial y}(\hat{y}^{0}(s),\hat{y}(s))\hat{w}^{j}_{i}(s)\right)\right]\,\,\,\,
   \; \vspace{0.3cm}\\ a.e.\; s\in [0,{\hat S}]$

 \item[$(iii)\;$] \begin{small}
 $\ds \left(1+d\right)\sum_{i=1}^{n+3}a^{i}\Big[\lambda^{0}(s)w_{i}^{0}+\lambda(s)\cdot\Big(\hat f(s)w_{i}^{0}+\sum_{j=1}^{m}\hat g_{j}(s)w_{i}^{j}\Big)+  \beta |w_{i}| \Big]\\  
 \leq \ds\sum_{i=1}^{n+3}\hat{a}^{i}(s)\Big[ \lambda^{0}(s)\hat{w}^{0}_{i}(s)+\lambda(s)\cdot\Big(\hat f(s)\hat{w}^{0}_{i}(s)+\sum_{j=1}^{m}\hat g_{j}(s)\hat{w}^{j}_{i}(s)\Big)+  \beta|\hat{w}_{i}(s)|\Big]  $
\end{small}

 for every $(w^{0},w,d)\in \mathbf{W}\times[-0.5,0.5]\,$ and $\,s\in [0,{\hat S}]\backslash I_0$\footnote{We have set $\hat{f}(s):= f(\hat{y}^{0}(s),\hat{y}(s))$, $\hat{g}_j(s):= g_{j}(\hat{y}^{0}(s),\hat{y}(s))$, for all $s\in [0,\hat S]$ and $i=1,\ldots,m$}
 
\item[$(iv)\;$] $(\lambda^{0}({\hat S}),\lambda({\hat S}),\beta)\in- {C}^{\perp}.$
 \end{itemize}
\end{thm}

 \subsection{{\sc Case (ii)} (Space-time  non-convex extension)}

Let us recall that we are assuming that 
\eqref{conic}, namely
$$  (r,u)\in [0,+\infty[\times U \implies \exists \rho >r \ \hbox{s.t.} \ \rho u\in U.
$$

Unlike the previous case, we are not going to convexify the dynamics, while we will consider just the impulsive extension.  Without repeating all steps, we just observe that the sought extension is obtained by
neglecting the sets $A$ and $\mathcal{A}$, and by replacing 
$\mathbf{W}$ with the (generally non-convex) set $ \mathbf{W}^{nc} := 
\{ (w^{0},w)\in [0,\infty)\times{U}:\quad w^{0}+|w|=1\},
$
respectively.  In turn, problem $(P)_\W^h$ simplifies into the following non-convex problem $(P^{nc})_{\W}^{h}$:

\bel{spnc2}
\,(P^{nc})_{\W}^{h}\left\{ \begin{array}{l}
\,\,\,\,\,\,\,\,\,\,\,\,\,\,\,\,\,\mathrm{Minimize}
\,h\big(z^{0}(\hat S), z(\hat S),\nu(\hat S)\big)\vspace{0.3cm}\\

\mathrm{over\:}\;(z^{0}, z, \nu, (w^{0},w),d)\in  AC([0,\hat S],\R\times\N\times\R)\times {\W}  \times \mathcal{D}
\quad\vspace{0.3cm}\\\mathrm{s.\: t., \,\,\,for\,\, a.e. \,\,\,\,s\in[0,\hat S], }\\\,\\
\left\{\begin{array}{l}
\ds\frac{dz^{0}}{ds}(s)= (1+d(s))w^{0}(s) \qquad\\
\ds\frac{dz}{ds}(s)=(1+d(s))\Big(f(z^{0}(s),z(s))w^{0}(s)+\sum_{j=1}^{m}g_{j}(z^{0}(s),z(s))w^{j}(s)\Big)\quad\qquad\\
\ds\frac{d\nu}{ds}(s)=(1+d(s))|w(s)| \qquad\end{array}\right.\vspace{0.3cm}\\

(z^{0}(0), z(0),\nu(0))=(0, \bar{x}, 0),\qquad (z^{0}(\hat S), z(\hat S), \nu(\hat S)) \in \bar{\cruk}\times [0,K]

\end{array}\right. . \\
\eeq

The other objects simplify accordingly, and, still because of the concatenation property, {\it the resulting family $\V$ is abundant in the resulting  $\W$:} 
 Therefore, by applying the infimum-gap result stated in Corollary \ref{MP_INF_GAP} we get:

\begin{thm}\label{gap_nc}
Consider a feasible extended process $$\Big(\big(y^{0},y,\nu\big)\ , \ \big((\hat w^{0},\hat w), \hat d \big)\Big), \quad \hat d \equiv 0,$$ and assume that it satisfies the infimum gap condition.  Then, for all QDQ  approximating cones     $C$  to $\cruk :=\bar{\cruk}\times [0,K]$ at $(\hat{y}^{0}, \hat{y}, \hat{\nu})({\hat S})$,    there exist a number $\beta\leq0$,
 an absolutely continuous    path $(\lambda^{0}, \lambda, \lambda^{\nu})\in  W^{1,1}([0,{\hat S}];\,\rr^{(1+n+1)})$ and a zero-measure subset $I_0$  such that the following conditions hold true:

\begin{itemize}
\item[$(i)\;$] $(\lambda^{0}, \lambda, \lambda^{\nu})\neq 0  ;\qquad$
\item[$(ii)\;$] $ \ds \frac{d\lambda^{0}}{ds}(s)= -\lambda(s)\cdot\left(\frac{\partial f}{\partial y^{0}}(\hat{y}^{0}(s),\hat{y}(s))\hat{w}^{0}(s)+\sum_{j=1}^{m}\frac{\partial g_{j}}{\partial y^{0}}(\hat{y}^{0}(s),\hat{y}(s))\hat{w}^{j}(s)\right)$

 $ \ds \frac{d\lambda}{ds}(s)= -\lambda(s)\cdot\left(\frac{\partial f}{\partial y}(\hat{y}^{0}(s),\hat{y}(s))\hat{w}^{0}_{i}(s)+\sum_{j=1}^{m}\frac{\partial g_{j}}{\partial y}(\hat{y}^{0}(s),\hat{y}(s))\hat{w}^{j}_{i}(s)\right),\,\,\,\,
   \; \vspace{0.3cm}$
   
    for a.e. $s\in [0,{\hat S}]$;


 \item[$(iii)\;$] \begin{small}
 $\ds \left(1+d\right)\Big[\lambda^{0}(s)w^{0}+\lambda(s)\cdot\Big(\hat f(s)w^{0}+\sum_{j=1}^{m}\hat g_{j}(s)w^{j}\Big)+  \beta |w| \Big]$ 
 
 $\qquad\qquad\qquad\qquad\qquad
 \leq \ds\Big[ \lambda^{0}(s)\hat{w}^{0}(s)+\lambda(s)\cdot\Big(\hat f(s)\hat{w}^{0}(s)+\sum_{j=1}^{m}\hat g_{j}(s)\hat{w}^{j}(s)\Big)+  \beta|\hat{w}(s)|\Big]  $
\end{small}

 for every $(w^{0},w,d)\in \mathbf{W}^{nc}\times[-0.5,0.5]\,$ and $\,s\in [0,{\hat S}]\backslash I_0$
 
\item[$(iv)\;$] $(\lambda^{0}({\hat S}),\lambda({\hat S}),\beta)\in- {C}^{\perp}.$

 \end{itemize}
\end{thm}

\appendix 
\section{Appendix}


\subsection{An example  on why abundance is crucial}
\label{ex_suss} The following example,  which is due to H.J. Sussmann,\footnote{Personal communication.} shows  how  the abundance hypothesis plays  crucial for the validity of Theorem \ref{mainth}.

Consider the families 
of controls $\V\subset\W$ defined as
$$
 \mathcal{W}:= L^1([0,1],[0,5])
 \qquad\mathcal{V}:=\left\{v\in \mathcal{W}:\; \int_{0}^{1}v(s)ds \neq 1 \right\},$$
 and  the optimal control problems
\[
(P)_{\mathcal{V}}\,\left\{ \begin{array}{l}
\mathrm{Minimize}\;  y(1)\,\,\\
\mathrm{over\:  processes\:}(y, v)(\cdot) \in  W^{1,1}([0,1],\R)  \times \mathcal{V}\quad  \\
\ds\frac{dy}{ds}(s)=v(s),\qquad\mathrm{a.e.\,\, s\in [0,1]}\qquad\\
y(0)=0, \qquad y(1)=1.
\end{array}\right. ,
\]
\[
(P)_{\mathcal{W}}\,\left\{ \begin{array}{l}
\mathrm{Minimize}\;  y(1)\,\,\\
\mathrm{over\:  processes\:}(y, w)(\cdot) \in  W^{1,1}([0,1],\R)  \times \mathcal{W}\quad  \\
\ds\frac{dy}{ds}(s)=w(s),\qquad\mathrm{a.e.\,\, s\in [0,1]}\qquad\\
y(0)=0, \qquad y(1)=1.
\end{array}\right. ,
\]
The process $(\hat{y}, \hat{w})(s) := (s,1)$ is a minimizer of the extended problem $(P)_\W$, with cost equal to $1$. If we restrict the controls to the {\it original} family of controls $\mathcal{V}$, 
 the cost of the problem raises to  $+\infty$, since  every solution $y[v]$ with $v\in \V$ fails to be feasible. In other words  the process $ (\hat{y}, \hat{w})$ satisfies the infimum gap condition.

By  applying the Pontryagin's Maximum Principle to the minimizer $(\hat{y}, \hat{w})$ of $(P)_\W$, we get that  there exist multipliers $(\lambda(\cdot), \lambda_{c})\neq (0,0)$ such that 
$$\ds\frac{d\lambda}{ds}(s)\equiv 0,\qquad \lambda(s) \w\leq \lambda(s) \qquad \forall \w\in[0,5], \quad s\in[0,1].$$
In particular this implies $\lambda(s) \equiv 0$ and $\lambda_{c}>0$. Therefore,  if we set $h(y):=y$ for every $y\in\R$,
the process $(\hat{y}, \hat{w})$   turns out to be a 
{\it normal $h$-extremal}.
Therefore, in view of Corollary \ref{MP_INF_GAP} the set $\mathcal{V}$,   though being dense in $\W$,  cannot be abundant in $\W$.  As a matter of fact, one can easily find a positive integer $N$, $\delta>0$ and $N+1$ controls $w, w_1,\ldots, w_N$ for which  $\theta^{\delta}_{w, w_1,\ldots,w_N}: \Gamma_N \rightarrow \mathcal{V}$ verifying the properties of  Definition \ref{abundant} does not exist.  Indeed, consider  $\Gamma_1(=[0,1])$, 
 $w(s):=0, w_1(s):=2$, $\forall s\in [0,1]$, 
$\delta>0$, and take any  mapping $\theta^{\delta}: [0,1]\rightarrow \mathcal{V}$. In view  of Definition \ref{abundant},  one has 
$$\ds \lim_{\delta\to 0} \int_{0}^{1}\theta^{\delta}(\gamma)(s)ds =  \int_{0}^{1}  w(s)+\gamma\big(w_1(s)-w(s)\big)=  2\gamma\qquad \forall \gamma\in \Gamma_1.$$

Then, for every $\delta$ sufficiently small, either there exists a $\gamma_{\delta}\in[0,1]$ such that 
$$\int_{0}^{1}\theta^{ \delta}(\gamma_{\delta})(s)ds=1,$$
or the map $\gamma \mapsto \int_{0}^{1}\theta^{\delta}(\gamma)(s)ds$ is not continuous. Since the  former case is ruled out by the fact that  the map $\theta^{\delta}(\cdot)(s)$  has to take values in $\mathcal{V}$, the map  $\gamma \mapsto \int_{0}^{1}\theta^{\delta}(\gamma)(s)ds$ is not continuous, so providing a contradiction.

%
%

\end{document}